\newif\ifdcg
\dcgfalse

\ifdcg
\RequirePackage{fix-cm}
\documentclass[smallextended,numbook,envcountsame,envcountreset,nospthms]{svjour3}
\usepackage{graphicx}
\usepackage{amsfonts,amsmath,amssymb,amscd,enumitem,color}
\usepackage{mathtools}
\usepackage[hidelinks]{hyperref}
\makeatletter
\let\cl@chapter\undefined
\makeatletter
\usepackage[english]{babel}
\usepackage[utf8]{inputenc}

\usepackage{amsthm}
\usepackage[capitalize]{cleveref}
\newtheorem{theorem}{Theorem}
\newtheorem{proposition}[theorem]{Proposition}
\newtheorem{definition}[theorem]{Definition}
\newtheorem{lemma}[theorem]{Lemma}
\newtheorem{problem}[theorem]{Problem}
\newtheorem{conjecture}[theorem]{Conjecture}
\newtheorem{example}[theorem]{Example}
\numberwithin{theorem}{section}
\newtheorem{corollary}[theorem]{Corollary}
\newtheorem{question}[theorem]{Question}
\newtheorem{assumption}[theorem]{Assumption}

\author{Brett Leroux \and Luis Rademacher}

\institute{Brett Leroux \at
              University of California, Davis, CA, USA \\
              \email{beleroux@ucdavis.edu}           %  \\
           \and
           Luis Rademacher \at
              University of California, Davis, CA, USA \\
              \email{lrademac@ucdavis.edu}
}

\date{Received: date / Accepted: date}

\else

\documentclass[11pt]{article}
\usepackage{graphicx}
\usepackage{amsfonts,amsmath,amssymb,amscd,enumitem,color}
\usepackage{mathtools}
\usepackage{accents}
\usepackage{amsthm}
\usepackage[hidelinks]{hyperref}
\usepackage[capitalize]{cleveref}
\usepackage[english]{babel}
\usepackage[utf8]{inputenc}
\usepackage{hyperref}
 
\newtheorem{theorem}{Theorem}
\newtheorem{proposition}[theorem]{Proposition}
\newtheorem{definition}[theorem]{Definition}
\newtheorem{corollary}[theorem]{Corollary}
\newtheorem{lemma}[theorem]{Lemma}

\newtheorem{question}[theorem]{Question}

\numberwithin{subcase}{case}

\newtheorem{assumption}[theorem]{Assumption}
\numberwithin{subsubcase}{subcase}

\numberwithin{theorem}{section}

\usepackage{fullpage}

\newenvironment{acknowledgements}{\paragraph{Acknowledgements}}{}

\author{Brett Leroux, Luis Rademacher}

\def\keywords#1{\par\addvspace\medskipamount{\rightskip=0pt plus1cm
\def\and{\ifhmode\unskip\nobreak\fi\ $\cdot$
}\noindent{Keywords:}\enspace\ignorespaces#1\par}}

\fi

\newcommand{\area}{\operatorname{area}}
\newcommand{\gap}{(c \log n)}
\newcommand{\length}{\operatorname{length}}
\newcommand{\one}{\mathbf{1}}
\newcommand{\ZZ}{\mathbb{Z}}
\newcommand{\giventhat}{\mid}
\newcommand{\pdist}{P}
\newcommand{\pset}{S}

\title{Improved bounds for the expected number of $k$-sets}

\usepackage{dirtytalk}

\def\final{1}  % set this to 1 to get a comment-free version

\ifnum\final=0  %namely if we allow comments in the output
\newcommand{\lnote}[1]{[{\small Luis: \bf #1}]}
\newcommand{\bnote}[1]{[{\small Brett: \bf #1}]}
\newcommand{\anonnote}[1]{[{\small anon: \bf #1}]}
\newcommand{\sidecomment}[1]{\marginpar{\tiny #1}}
\newcommand{\details}[1]{{\color{blue}\ [[#1]] }}
\else % in this case [final=1] we don't want any comments to show
\newcommand{\lnote}[1]{}
\newcommand{\bnote}[1]{}
\newcommand{\anonnote}[1]{}
\newcommand{\sidecomment}[1]{}
\newcommand{\details}[1]{}
\fi  %ok, we're done with defining comment macros
\newcommand{\Rl}{\operatorname{\mathbb{R}}}

\newcommand{\RR}{\mathbb{R}}

\newcommand{\Tor}{\operatorname{\Tor}}

\newcommand{\aff}{\operatorname{aff}}

\newcommand{\conv}{\operatorname{conv}}
\newcommand{\interior}{\operatorname{int}}
\newcommand{\suchthat}{\mathrel{:}}
\newcommand{\bd}{\operatorname{bd}}

\newcommand{\e}{\operatorname{\mathbb{E}}}
\newcommand{\pr}{\operatorname{\mathbb{P}}}
\newcommand{\ud}{\mathrm{d}}

\newcommand{\floor}[1]{\lfloor{#1}\rfloor}
\newcommand{\norm}[1]{\lVert{#1}\rVert}
\newcommand{\abs}[1]{\lvert#1\rvert}
\newcommand{\card}[1]{\lvert#1\rvert}
\newcommand{\eps}{\epsilon}
\newcommand{\cl}{\operatorname{cl}}
\newcommand{\fC}{\operatorname{\mathcal{C}}}

\makeatletter
\newcommand\fake@math{}% just for safety
\def\fake@math#1\){[math]}
\makeatother

\begin{document}

\maketitle

\begin{abstract}
Given a finite set of points $S\subset\mathbb{R}^d$, a \emph{$k$-set} of $S$ is a subset $A \subset S$ of size $k$ which can be strictly separated from $S \setminus A $ by a hyperplane. 
Similarly, a \emph{$k$-facet} of a point set $S$ in general position is a subset $\Delta\subset S$ of size $d$ such that the hyperplane spanned by $\Delta$ has $k$ points from $S$ on one side.
For a probability distribution $P$ on $\mathbb{R}^d$, we study $E_P(k,n)$, the expected number of $k$-facets of a sample of $n$ random points from $P$. 
When $P$ is a distribution on $\mathbb{R}^2$ such that the measure of every line is 0, we show that $E_P(k,n) = O(n(k+1)^{1/4})$.
Our argument is based on a technique by B\'{a}r\'{a}ny and Steiger.
We study how it may be possible to improve this bound using the continuous version of the polynomial partitioning theorem.
This motivates a question concerning the points of intersection of an algebraic curve and the $k$-edge graph of a set of points.

We also study a variation on the $k$-set problem for the set system whose set of ranges consists of all translations of some strictly convex body in the plane. 
The motivation is to show that the technique by B\'{a}r\'{a}ny and Steiger is tight for a natural family of set systems.
For any such set system, we determine bounds for the expected number of $k$-sets which are tight up to logarithmic factors. 
%If $P$ is a probability distribution on $\Rl^2$ which has a density, we show that the expected number of $k$-sets of a sample of $n$ random points from $P$ is $O(n^{5/4})$ when $k=\Theta(n)=\Theta(n-k)$.
\end{abstract}

\section{Introduction}
Let $S$ be a finite set of points in $\Rl^d$. A subset $A$ of $S$ of size $k$ is a \emph{$k$-set} of $S$ if there exists an affine hyperplane which strictly separates $A$ from $S \setminus A$. 
We use $a_d(k,n)$ to denote the maximal number of $k$-sets over all configurations of $n$ points in $\Rl^d$. The \emph{$k$-set problem} asks one to determine the asymptotic behavior of $a_d(k,n)$.

The $k$-set problem in the plane was first raised by A.~Simmons in unpublished work. Straus, also in unpublished work, gave a construction showing an $\Omega(n\log n)$ lower bound. Lov\'{a}sz \cite{lovasz1971number} published the first paper on $k$-sets, establishing an $O(n^{3/2})$ upper bound. See also the paper \cite{MR0363986} of Erd\H{o}s, Lov\'{a}sz, Simmons,  and Straus. Despite the efforts of many, the $k$-set problem is still not well understood even in the plane. The best known bounds are $a_2(k,n) = n e^{\Omega(\sqrt{\log k})}$ \cite{DBLP:conf/compgeom/Toth00,Tothlower,MR2405686} and $a_2(k,n) = O(n(k+1)^{1/3})$ \cite{Dey1997ImprovedBO,DBLP:journals/dcg/Dey98}. See \cite{WagnerSurvey} for an overview of results in higher dimensions. 

When studying the $k$-set problem, one usually only considers point sets which are in general position. 
\begin{definition}\label{def:genlinear}
A set of at least $d+1$ points in $\Rl^d$ is in \emph{general linear position} if no $d+1$ (and thus, fewer) points are affinely dependent. 
\end{definition}
When we say general position, we always mean general linear position. This reduction is justified by the observation that the maximum number of $k$-sets is attained by a set of points in general position (see for example \cite{WagnerSurvey}). 
For point sets in general position, one can study the closely related concept of \emph{$k$-facets}.

Let $S$ be a set of $n$ points in general position in $\Rl^d$. A subset $\Delta \subset S$ of size $d$ is a \emph{$k$-facet} of $S$ if the open halfspace on one side of $\aff \Delta$ contains exactly $k$ points from $S$. We use $e_d(k,n)$ to denote the maximal number of $k$-facets over all configurations of $n$ points in general position in $\Rl^d$. It can be shown that for fixed $d$, the functions $e_d(k,n)$ and $a_d(k,n)$ have the same asymptotic behavior \cite{WagnerSurvey}. In this paper, it is more convenient for us to work with $k$-facets. Furthermore, our main results are for the planar version of the $k$-facet problem. In the plane, $k$-facets are also known as \emph{$k$-edges}. 

It is widely believed that the true value of $e_2(k,n)$ is closer to the best known lower bound than to the best known upper bound. Indeed, Erd\H{o}s et al. conjectured in \cite{MR0363986} that $e_2(k,n) = O(n^{1+\epsilon})$ for any $\epsilon>0$. Some support for this conjecture is provided by the  results one can obtain for the probabilistic version of the $k$-facet problem which was originally studied by B\'{a}r\'{a}ny and Steiger in \cite{BaranySteiger}.
Given a probability distribution $P$ on $\Rl^d$, what is the expected number $E_P(k,n)$ of $k$-facets of $X$, a sample of $n$ independent random points from $P$?
Recall that the $k$-facet problem is only defined for point sets in general position. 
For this reason, in all of our results concerning $E_P(k,n)$, our attention is restricted to distributions $P$ such that the measure of every hyperplane is~0.
This is the minimal assumption on distributions $P$ which guarantees that a sample of points from $P$ is in general position with probability 1. We refer to the original $k$-facet problem as the \emph{deterministic} version. 
By the \emph{probabilistic} version of the $k$-facet problem, we mean the question of the value of $E_P(k,n)$. 

B\'{a}r\'{a}ny and Steiger showed in \cite{BaranySteiger} that $E_P(k,n) = O(n^{d-1})$ if $P$ is spherically symmetric. 
Also, if $P$ is the uniform  distribution on a convex body in $\Rl^2$, they show that $E_P(k,n) =O(n)$. 
Similar sorts of bounds were obtained in \cite{clarkson} and \cite{MR4194871}: Clarkson showed in \cite{clarkson} that $E_P(k,n) = O(n^{d-1})$ if $P$ is a coordinate-wise independent distribution. A result of Chiu et al. (\cite[Theorem 3]{MR4194871}) improves the bound obtained by B\'{a}r\'{a}ny and Steiger when $P$ is a specific type of spherically symmetric distribution. These results give some evidence that the conjecture of Erd\H{o}s et al. that $e_2(k,n) = O(n^{1+\epsilon})$ for any $\epsilon>0$ (and the extension to higher dimensions, i.e., $e_d(k,n) = O(n^{d-1+\epsilon})$ for any $\epsilon>0$) may be true.

However, when one considers more general probability distributions, the asymptotic behavior of the expected number of $k$-edges becomes more difficult to determine. In fact, as noted in \cite[Section 4.2]{WagnerSurvey}, there is some reason to believe that the probabilistic version of the $k$-facet problem may be more or less the same as the original: Circa publication of  B\'{a}r\'{a}ny and Steiger's paper \cite{BaranySteiger}, the best known lower bound for the deterministic $k$-facet problem in $\Rl^2$ was $e_d( \frac{n-2}{2},n) = \Omega(n \log n)$ \cite{MR0363986}. 
Using the construction in \cite{MR0363986}, B\'{a}r\'{a}ny and Steiger construct a probability distribution $P$ with $E_P(\frac{n-2}{2},n) = \Omega(n\log n)$. The $\Omega(n \log n)$ lower bound for the deterministic $k$-facet problem was improved to $e_d( \frac{n-2}{2},n)=ne^{\Omega(\sqrt{\log n})}$ in \cite{Tothlower}. As noted in \cite{WagnerSurvey}, it is possible to use the construction in \cite{Tothlower} to construct a distribution $P'$ with $E_{P'}(\frac{n-2}{2},n) =ne^{\Omega(\sqrt{\log n})}$. See \cite{WagnerSurvey} for some more details. 

\section*{Main results}
Our main results are summarized below, and described in more detail in the rest of the introduction.
\begin{itemize}
\item
\textbf{Expected number of $k$-edges.} For any Borel probability $P$ on $\Rl^2$ such that the measure of every line is 0, we show that the expected number of $k$-edges $E_P(k,n)$ of a sample of $n$ points is $O(n(k+1)^{1/4})$ (\cref{thm:5/4}). This is an improvement to the best known bound for the maximum number of $k$-edges of an $n$ point set in the plane which is the $O(n(k+1)^{1/3})$ by Dey \cite{Dey1997ImprovedBO,DBLP:journals/dcg/Dey98}.

\item
\textbf{Translations of a fixed convex body in the plane.} We study a natural variation on the $k$-set problem where the separating curves are translations of the boundary of some strictly convex body in the plane instead of lines as in the original $k$-set problem. Let $C \subset \Rl^2$ be some strictly convex body. For a set $S\subset \Rl^2$ of $n$ points, a \emph{$T_C$-$k$-set} of $S$ is a subset $T \subset S$ of size $k$ such that there exists a translation of $C$ which contains $T$ in its interior and contains no point in $S \setminus T$ (\cref{def:T_C}). For certain distributions on $\Rl^2$, we show that the expected number of $T_C$-$k$-sets of a sample of $n$ random points is $\tilde\Theta(n^{3/2})$ (where $\tilde\ $ means that polylogarithmic factors are ignored) (\cref{thm:halvingaverage,thm:vcavg}). 

\end{itemize}

\subsection{Expected number of \texorpdfstring{$k$}{k}-edges}

B\'{a}r\'{a}ny and Steiger studied the expected number of $k$-facets using an integral formula for $E_P(k,n)$. 
We also use the formula in the proofs of our main results so we describe it here. 
Let $X_1, \dotsc, X_d$ be $d$ random points drawn from $P$. 
The assumption that the measure of every hyperplane is zero implies that $\aff(X_1,\dots,X_d)$ is a hyperplane with probability 1. 
The same assumption implies that hyperplane $\aff(X_1,\dots,X_d)$ does not contain a line parallel to the $d$th standard basis vector $e_d$ with probability 1.
Therefore, choosing the direction of $e_d$ as the \say{up} direction, we can use $\aff(X_1,\dots,X_d)^+$ to denote the open half space above $\aff(X_1,\dots,X_d)$, i.e., in the direction of $e_d$. 
Similarly, we use $\aff(X_1,\dots,X_d)^-$ to denote the open half space below $\aff(X_1,\dots,X_d)$, i.e., in the direction of $-e_d$.
Define
\[
G_P(t) = \mathbb{P}\big(P(\aff(X_1,\dots,X_d)^+)\le t\big).
\]
Then given a sample $X$ of $n$ random points from $P$, the expected number of $k$-facets of $X$ is
\begin{equation}\label{eq:BaranySteiger}
\begin{split}
E_P(k,n) &=  \sum_{F \in \binom{X}{d}} \mathbb{P} \bigl(\aff(F)^+ \text{ or } \aff(F)^- \text{ contains exactly }k \text{ points of } X  \bigr) \\
&= 2\binom{n}{d}\binom{n-d}{k} \int_0^1 t^k(1-t)^{n-d-k} \ud G_P(t). 
\end{split}
\end{equation}

First, we observe that \cref{eq:BaranySteiger} can be used to immediately obtain the following upper bound for the expected number of $\frac{n-d}{2}$-facets of a sample of $n$ points from \emph{any}\footnote{Recall we assume that hyperplanes have measure 0. In all our results we also restrict to Borel measures. The Borel assumption guarantees that every open halfspace is measurable which is necessary for \cref{eq:BaranySteiger} (and, in particular, the definition of $G_P(t)$) to make sense. In fact the assumption that $P$ is Borel is the minimal assumption that guarantees that all halfspaces are measurable because the collection of open halfspaces generates the Borel sigma algebra on $\Rl^d$.} probability distribution on $\Rl^d$. 
The bound is much weaker than the bounds obtained by B\'{a}r\'{a}ny and Steiger for special cases of the distribution, but it is still non-trivial and also applies to all distributions (not only
the spherically symmetric ones).
\begin{theorem}\label{thm:basic}
If $P$ is a Borel probability distribution on $\Rl^d$ such that the measure of any hyperplane is zero, then $E_P(\frac{n-d}{2},n)= O(n^{d-1/2})$ (where the constants in big-$O$ depend only on $d$).
\end{theorem}
\begin{proof}
From \cref{eq:BaranySteiger} and Stirling's approximation,
\begin{equation}\label{eq:generalbound}
\begin{split}
E_P \left(\frac{n-d}{2},n \right) &= 2\binom{n}{d}\binom{n-d}{\frac{n-d}{2}} \int_0^1 t^{\frac{n-d}{2}}(1-t)^{\frac{n-d}{2}} \ud G_P(t)\\
& \le 2\binom{n}{d}\binom{n-d}{\frac{n-d}{2}} \frac{1}{2^{n-d}} \\
& = O(n^{d-1/2}).
\end{split}
\end{equation}
(The above inequality holds because $t^a (1-t)^b$ is maximized when $t=a/(a+b)$ if $a,b \neq 0$.)
\end{proof}
In \cref{sec:2} we use the weak bound in the proof of \cref{thm:basic} combined with a partition of the plane to show an improved bound for the expected number of $k$-edges:

\begin{theorem}\label{thm:5/4}
Let $P$ be a Borel probability distribution on $\Rl^2$ such that the measure of every line is zero. Then $E_P(k,n) \le 10n(k+1)^{1/4}$. 
\end{theorem}

The proof of \cref{thm:5/4} is in \cref{sec:pflines}. In \cref{sec:convex/concave}, we review the definition of the $k$-edge graph (\cref{def:graph}) and state some results needed for the proof of \cref{thm:5/4}. In particular, we explain how the $k$-edge graph can be decomposed into \emph{convex chains}. The proof of \cref{thm:5/4} also uses a divide and conquer approach where the plane is partitioned into cells by vertical lines. 

\cref{sec:3} outlines how it may be possible to improve the bound in \cref{thm:5/4} by partitioning the plane with algebraic curves rather than vertical lines. 
The existence of algebraic curves which partition the plane in a useful way is a consequence of the continuous version of the polynomial partitioning theorem of \cite{Guth} (\cref{thm:non-singularpartition}).
Whether or not using algebraic curves rather than lines leads to an improvement in the bound depends on a question we leave open (\cref{question}) which asks for a bound on the maximum number of times that an algebraic curve of degree $r$ can intersect the $k$-edge graph of a set of $n$ points. 
We show that this quantity is $O(nr^2)$ but, as far as the authors know, it may be possible to improve this bound, see \cref{sec:conclusion}.

\subsection{Tightness of the argument in \texorpdfstring{\cref{thm:basic}}{Theorem \ref{thm:basic}}: Translations of a fixed shape}

In \cref{sec:tight}, we show that the two-step argument in the proof of \cref{thm:basic} (\cref{eq:BaranySteiger} and the general upper bound of the integrand in \cref{eq:generalbound}) is not as loose as it seems, if one applies it to the $k$-set problem on a set system other than half-planes (generalized in a natural way).
More precisely, let $C \subseteq \RR^2$ be the interior of a fixed convex body. 
We consider the set system of translations of $C$ and study the expected number of ways in which one can enclose $k$ points out of a given finite set of points using translations of $C$.
Our variations of $k$-sets and $k$-edges for this problem are called $T_C$-$k$-sets and $T_C$-$k$-edges (\cref{def:T_C}).
In addition, we show some deterministic bounds to put our probabilistic bounds in context.

For the case where $C$ is strictly convex, we show:
\begin{itemize}
\item A relation between $T_C$-$k$-sets and $T_C$-$k$-edges that allows one to derive upper bounds on the number of $T_C$-$k$-sets from upper bounds on the number of  $T_C$-$k$-edges (\cref{lem:shp}).
\item For certain natural distributions, the expected number of $T_C$-$k$-sets and $T_C$-$k$-edges for a random set of $n$ points and some $k$ proportional to $n$ is $\tilde\Theta(n^{3/2})$ (\cref{thm:halvingaverage,thm:vcavg}).
The upper bound uses the B\'{a}r\'{a}ny and Steiger technique, while the lower bound uses the uniform convergence theorem of Vapnik and Chervonenkis \cite{VapChe71}.
\item The growth function is $O(n^2)$ (\cref{cor:growth}).
\end{itemize}
For the case where $C$ has $C^2$ boundary (\cref{def:C2}), we show that the maximum number of $T_C$-$k$--sets of $n$ points with $k$ proportional to $n$ is $\Omega(n^2)$ (\cref{thm:c2cross}).

Some of the assumptions above are chosen for readability, the actual theorems have weaker assumptions in some cases.

\section{Bounding the expected number of \texorpdfstring{$k$}{k}-edges}\label{sec:2}
In this section we study the original probabilistic $k$-set/$k$-facet problem in the plane and prove \cref{thm:5/4}.

\subsection{Convex/concave chains}\label{sec:convex/concave}

Here we recall the \say{convex chains} technique of \cite{MR1608872} which was used in \cite{Dey1997ImprovedBO} to establish the $O(n(k+1)^{1/3})$ bound for planar $k$-edges. First we need to define the \emph{$k$-edge graph}.

Let $S$ be a set of $n$ points in general position in the plane and choose some $(x,y)$ coordinate system. With this choice, we assume without loss of generality that no line spanned by two points in $S$ is vertical. Let $E_k$ be the set of line segments connecting two points $x,y \in S$ such that there are exactly $k$ points from $S$ in the (open) halfplane below $\aff(x,y)$. Therefore, $E_k$ is a subset of the set of all $k$-edges of $S$. 
%Throughout, we assume without loss of generality that $E_k$ contains at least half the total number of $k$-edges. Indeed, we could repeat the analysis after rotating the plane 180 degrees. 
The line segments $E_k$ define the $k$-edge graph:
\begin{definition}\label{def:graph}
Let $S$ be a set of $n$ points in general position in the plane and assume that no line spanned by two points in $S$ is vertical. For any $0 \le k \le n-2$, the (geometric) graph $G_k = (S,E_k)$ is called the \emph{$k$-edge graph} of $S$. 
\end{definition}

The convex chains technique decomposes the $k$-edge graph $G_k$ of a point set $S$ into the union of a bounded number of \emph{convex chains}. 
Each convex chain is the graph of a convex piece-wise linear function defined on some interval of the $x$-axis. 
Each chain is formed by some subset of the $k$-edges in $G_k= (S,E_k)$.  
A simpler version of the proof of the $O(n(k+1)^{1/3})$ bound was established in \cite{Sariel} by observing that the $k$-edge graph can simultaneously be decomposed into the union of \emph{concave chains}. 
%We use the known fact \cite{lovasz1971number} that the convex/concave chain decompositions imply that a vertical line can intersect the $k$-edge graph at most $\min(k+1,n-k-1)$ times. 
% Our proof in \cref{sec:curves} uses the convex and concave chain partitions simultaneously. 

%\lnote{the literature review may need more care; not clear what the origin of convex-concave chains argument is, it is mentioned in \url{https://jeffe.cs.illinois.edu/open/half.html}, probably before the book?. Also that page gives some credit to Agarwal Aronov Sharir.}

\begin{lemma}[{Convex/concave chains \cite{MR1608872}, \cite[Lemma 9.10]{Sariel}\footnote{\cite[Lemma 9.10]{Sariel} states that the number of convex chains is $k-1$ and the number of concave chains is $n-k+1$. The reason this doesn't match the Lemma as we have stated it is that \cite{Sariel} defines a $k$-edge to be line segment connecting two points $x,y$ such that there are $k$ points in the \emph{closed} halfspace below $\aff(x,y)$ whereas we require that there are $k$ points in the \emph{open} halfspace below $\aff(x,y)$. We choose open because it matches the standard definition of a $k$-facet.} }]\label{lem:convexchains}
Let $G_k= (S,E_k)$ be the $k$-edge graph of a set $S$ of $n$ points in the plane.
The graph $G_k$ can be decomposed into the union of $k+1$ (piece-wise linear) convex chains. 
Similarly, the graph can be decomposed into the union of $n-k-1$ (piece-wise linear) concave chains. 
\end{lemma}

%These decompositions can be used to show that the total number of crossings in $G_k$ is $O(n^2)$, see \cite{Dey1997ImprovedBO} or \cite{Sariel}.
\subsection{Proof of \texorpdfstring{\cref{thm:5/4}}{Theorem \ref{thm:5/4}} }\label{sec:pflines}
We are now ready to prove our bound on the expected number of $k$-edges. 
The idea of the proof is to use vertical lines to divide the plane into a number of regions of equal probability.
We then bound separately the expected number of $k$-edges that intersect one of the vertical lines and the expected number of $k$-edges that do not intersect any of the lines. We remark that a partition of the plane using vertical lines was also used by Lov\'{a}sz in \cite{lovasz1971number} to establish the $O(n^{3/2})$ bound for the deterministic $k$-set problem.

\begin{proof}[Proof of \cref{thm:5/4}]
Because of the symmetry of the $k$-edge problem, we can assume without loss of generality that $k \le (n-2)/2$. The conclusion is clearly true when $k=0$ so assume also that $k \ge 1$. 

We first bound the expected number of $k$-edges \emph{that are in $G_k$}. At the end of the proof we explain how this leads to a bound for the expected number of $k$-edges.

For any fixed $n$, let $m:=m(n)$ be an integer whose value will be chosen later and let $L$ be the set consisting of $m$ vertical lines which partition $\Rl^2 \setminus \big( \bigcup_{\ell \in L} \ell\big)$ into $m+1$ open cells such that the measure according to $P$ of each cell is equal to $1/(m+1)$. 
Such a set of lines exists because the measure (according to $P$) of every line is 0.\details{In order to show how to construct this set of lines, it suffices to show that given any finite measure $\mu$ on $\Rl^2$ such that the measure of every line is 0, there exists a vertical line which divides the measure into two equal parts. Indeed, the function $f$ defined by $f(t) = \mu((x,y) \in \Rl^2 \suchthat x \le t)$ is a continuous function of $t$. The fact that $f$ is continuous is a consequence of the assumption that the measure of every vertical line is 0. Furthermore, as $t \to -\infty$, $f(t) \to 0$ and as $t \to \infty$, $f(t) \to \mu(\Rl^2)$. Therefore, the claim follows from the intermediate value theorem.} 

Let $X = \{X_1, \dotsc, X_n\}$ be a sample of $n$ iid points from $P$. 
Observe that the probability that two points in $X$ span a vertical line is 0 so we do not need to consider this case.
Also, since the measure of every line is 0, $X$ is in general position with probability 1 so we can also ignore the case when $X$ is not in general position. Therefore, we can analyze the $k$-edges of $X$ using the $k$-edge graph $G_k$ of $X$.

We bound the expected number of $k$-edges in $G_k$ by considering two different types of $k$-edges separately. First we bound the expected number of $k$-edges formed by two points in different cells of the partition. Then we bound the expected number of $k$-edges formed by two points in the same cell of the partition. That is, the expected number of $k$-edges in $G_k$ is equal to
\begin{equation}\label{eq:decomp}
\begin{split}
& \mathbb{E}(\text{number of } k\text{-edges in }G_k \text{ formed by two points in different cells}) \\ 
&+\mathbb{E}(\text{number of } k\text{-edges in }G_k \text{ formed by two points in the same cell}).
\end{split}
\end{equation}
If $\conv(X_1,X_2)$ is a $k$-edge formed by two points $X_1,X_2$ in different cells, then $\conv(X_1,X_2)$ intersects at least one line in $L$. 
So to bound the expected number of $k$-edges in $G_k$ formed by two points in different cells, it suffices to bound the expected number of $k$-edges in $G_k$ that intersect a line in $L$. By \cref{lem:convexchains}, each line in $L$ intersects at most $\min(k+1,n-k-1) = k+1$ $k$-edges in $G_k$. Therefore, the first term in \cref{eq:decomp} is at most $m(k+1)$.

Now we bound the second term. Recall that the measure according to $P$ of each cell is equal to $1/(m+1)$. Therefore, for any fixed $i \neq j$, the probability that $X_i$ and $X_j$  are in the same cell is $1/(m+1) \le 1/m$. 

We can bound the second term in \cref{eq:decomp} by
\begin{equation}\label{eq:cond}
\begin{split}
&\mathbb{E}(\text{number of } k\text{-edges in }G_k \text{ formed by two points in the same cell}) \\
&\le\sum_{(X_i,X_j) \in \binom{X}{2}} \mathbb{P} \bigl( (X_i,X_j) \text{ is a }k\text{-edge AND } X_i,X_j \text{ are in the same cell} \bigr)\\ 
&= \sum_{(X_i,X_j) \in \binom{X}{2}} \mathbb{P} \bigl((X_i,X_j) \text{ is a }k\text{-edge } \bigm| X_i,X_j \text{ are in the same cell} \bigr) \\&\hspace{1.5cm} \cdot \mathbb{P}(X_i,X_j \text{ are in the same cell})\\
&= \binom{n}{2}\cdot \mathbb{P}\bigl((X_1,X_2) \text{ is a }k\text{-edge } \bigm| X_1 ,X_2 \text{ are in the same cell} \bigr) \\&\hspace{1.5cm}\cdot \mathbb{P}(X_1 ,X_2 \text{ are in the same cell}) \\
&\le \frac{n^2}{m}\cdot \mathbb{P} \bigl((X_1,X_2) \text{ is a }k\text{-edge } \bigm| X_1,X_2 \text{ are in the same cell} \bigr).
\end{split}
\end{equation}
Set $T:=P\bigl(\aff(X_1,X_2)^+\bigr)$ and $G_P(t) = \mathbb{P}(T\le t \giventhat X_1,X_2 \text{ are in the same cell})$.  
%Now let $x,y$ be two points chosen independently from $P$ restricted to one cell $O_i$ of the partition. 
%Let $\aff(x,y)^+$ denote the halfplane below $\aff(x,y)$. Set $G_P(t) = \mathbb{P}(P(\aff(x,y)^+)\le t)$. 
We then have that%
\footnote{We use the following version of the law of total probability:\\ $\mathbb{P}(A \giventhat Y) = \mathbb{E}\bigl(\mathbb{P}(A\giventhat X,Y) \giventhat Y\bigr)$. This follows from \cite[Theorem 4.1.13(ii)]{MR3930614}.
%We use the following version of the law of total expectation: $\mathbb{E}(X) = \mathbb{E}\bigl(\mathbb{E}(X\giventhat Y)\bigr)$, where $X$ is the indicator of the event in the left-hand side of \eqref{eq:int}.
}
\begin{equation}\label{eq:int}
\begin{split}
&\mathbb{P} \bigl(\text{$(X_1,X_2)$ is a $k$-edge} \bigm| X_1,X_2 \text{ are in the same cell} \bigr)\\
&= \mathbb{E}\Bigl( \mathbb{P} \bigl(\text{$(X_1,X_2)$ is a $k$-edge} \bigm| X_1,X_2,(X_1,X_2 \text{ are in the same cell}) \bigr) \\& \hspace{1.5cm}\Bigm| X_1,X_2 \text{ are in the same cell} \Bigr)\\
&= 2\binom{n-2}{k} \mathbb{E}\bigl( T^k(1-T)^{n-2-k} \bigm| X_1,X_2 \text{ are in the same cell} \bigr)\\
&= 2\binom{n-2}{k} \int_0^1 t^k(1-t)^{n-2-k} \ud G_P(t)\\
& \le 2\binom{n-2}{k}\cdot \bigg( \frac{k}{n-2} \bigg)^k \cdot  \bigg(\frac{n-2-k}{n-2} \bigg)^{n-2-k}\\
%\details{\sim\sqrt{\frac{1}{k} + \frac{1}{n-2-k}}\sqrt{2/\pi}}
& \le \frac{\sqrt{n-2}}{\sqrt{k}\sqrt{n-2-k}}.
\end{split}
\end{equation}
The second to last inequality holds because $t^a (1-t)^b$ is maximized when $t=a/(a+b)$ if $a,b \neq 0$. And the last inequality follows from Stirling-type upper and lower bounds for factorials, see for example \cite{Stirling}. Therefore, we have that \cref{eq:decomp} is at most
\[
m(k+1)+\frac{n^2}{m}\frac{\sqrt{n-2}}{\sqrt{k}\sqrt{n-2-k}}
\]
and choosing $m =\big\lfloor \frac{n}{k^{3/4}}\big\rfloor$ makes the above quantity less than $2(\sqrt{2}+1)n(k+1)^{1/4}$. 

%When $k =\floor{cn}$ for some $c \in (0,1)$, the $O\Big(\frac{n^{7/4}}{k^{1/4}(n-k)^{1/4}}\Big)$ bound just established  shows that the expected number of $k$-edges of a sample of $n$ points is $O(n^{5/4})$. 
%An artifact of this proof is that when $k$ grows much slower than $n$, the bound we obtain on the expected number of $k$-edges is worse. 
%For example, when $k$ is a constant and $n$ grows, the theorem only tells us that the expected number of $k$-edges is $O(n^{3/2})$. 
%However, we can combine our bound with the best known bound for the deterministic $k$-edge problem to obtain a bound on the expected number of $k$-edges that gives a uniform bound for all values of $k$. 
%Indeed if $P$ is any distribution on $\Rl^2$ such that the measure of every line is zero, then by the bound just established and the best known bound for the deterministic $k$-edge problem, i.e. \cite[Theorem 3.3]{DBLP:journals/dcg/Dey98}, the expected number of $k$-edges of a sample of $n$ points from $P$ is at most $\min(C_1nk^{1/3}, \frac{C_2n^{7/4}}{k^{1/4}(n-2-k)^{1/4}}  )$ for some constants $C_1,C_2$. We can choose the constant $C$ in the statement of the theorem as the maximum of $C_1,C_2$. 
%To show that this quantity is at most $2Cn^{9/7}$, first recall that because of the symmetry of the $k$-edge problem, we can assume without loss of generality that $k \le (n-2)/2$. When $k \le (n-2)/2$, and $k \le  n^{6/7}$, we have that $Cnk^{1/3} \le Cn^{9/7}$. And when $k \le (n-2)/2$ and $k \ge n^{6/7}$, we have $\frac{Cn^{7/4}}{k^{1/4}(n-2-k)^{1/4}}  \le \frac{2Cn^{3/2}}{n^{6/28}}  \le 2Cn^{9/7}$.

Observe that the number of $k$-edges of a point set $S$ is at most equal to the number of edges in the $k$-edge graph of $S$ plus the number of edges in the $k$-edge graph of the point set $S'$ where $S'$ is the point set obtained by rotating $S$ 180 degrees\footnote{When $k=\frac{n-2}{2}$ this counts each $k$-edge twice, so the constant in our bound can be improved in this case.}. Therefore, we can  rotate the plane 180 degrees and apply the same analysis to again bound the expected number of $k$-edges in~$G_k$. From this, we get that expected number of $k$-edges is at most two times the bound we obtain for the expected number of $k$-edges in $G_k$, which is less than $10n(k+1)^{1/4}$.
\end{proof}

\section{On the number of \texorpdfstring{$k$}{k}-edges via the polynomial method}\label{sec:3}

In this section we give another proof of \cref{thm:5/4} in the case when $k$ is proportional to $n$.
The new proof partitions the plane using algebraic curves instead of vertical lines. 
Recall that an algebraic curve in $\Rl^2$ is the set of zeroes of a polynomial equation in two variables, i.e., for a  polynomial $f \in \Rl[x_1, x_2]$, the \emph{algebraic curve} defined by $f$ is the set 
\[
    Z(f):= \{(x_1,x_2) \in \Rl^2  \suchthat f(x_1,x_2) = 0\}.
\] 

Given a distribution on $\Rl^2$ which has a density, we use the continuous polynomial partitioning theorem of \cite{Guth} (\cref{thm:non-singularpartition}) to obtain an algebraic curve which divides the plane into a number of \emph{cells} of equal probability. 
The rest of the proof is nearly the same as the proof in \cref{sec:pflines}. 
The reason this alternative proof is interesting is because it motivates the following open question which, if resolved, may lead to an improvement to the bound in \cref{thm:5/4}.

\begin{question}\label{question}
What is the maximum (finite) number of times that an irreducible non-singular\footnote{One could consider the same question for possibly singular curves, but, for our purposes, it suffices to consider non-singular curves.} degree $r$ algebraic curve can intersect the $k$-edge graph of a set of $n$ points in the plane?
\end{question}

It is clear that the quantity in \cref{question} is $\Omega(nr)$, and we have some reason to believe that it may be $\Theta(nr)$. The best bound we are able to prove is $O(nr^2)$ (\cref{lem:intersections}). 
This bound is good enough to reprove \cref{thm:5/4} using polynomial partitioning in the case where the distribution has a density (\cref{thm:5/4poly}). 
Any improvement to our $O(nr^2)$ bound would lead to an improvement in the bound in \cref{thm:5/4poly}: When $k$ is proportional to $n$, the bound in \cref{thm:5/4poly} is $O(n^{5/4})$. 
An $O(nr)$ bound on the quantity in \cref{question} would allow one to improve the bound in \cref{thm:5/4poly} from $O(n^{5/4})$ to $O(n^{7/6})$ when $k$ is proportional to $n$. The proof of the $O(n^{7/6})$ bound would be exactly the same as the proof of \cref{thm:5/4poly} except that the choice of the degree $r$ of the partitioning polynomial would be different. 

An $O(nr)$ bound on the quantity in \cref{question} would also have an interesting application to the deterministic $k$-set problem: It would give another proof of Dey's $O(n(k+1)^{1/3})$ bound \cite{Dey1997ImprovedBO,DBLP:journals/dcg/Dey98} on the maximum number of $k$-edges of a set of $n$ points in the plane in the case where $k$ is proportional to $n$. 
The idea of the proof is as follows. 
Given any set $S$ of $n$ points in general position in the plane, for some $r$ to be chosen later, use the discrete polynomial partitioning theorem (\cref{thm:discretepartition}) to find a degree $O(r)$ polynomial $f$ such that $\Rl^2 \setminus Z(f)$ is the union of $r^2$ pairwise disjoint open sets (called cells) each of which contains at most $n/r^2$ points of $S$. 
The $O(nr)$ bound on the quantity in \cref{question} implies that the number of $k$-edges formed by two points of $S$ which are both in different cells of the partition is $O(nr)$. 
Also, the number of $k$-edges formed by two points which are both in the same cell is at most $r^2 \cdot  \binom{n/r^2}{2} = O(\frac{n^2}{r^2})$.
Observe that we can apply a sufficiently small perturbation to the points of $S$ to obtain a point set which is in general position with respect to degree $O(r)$ algebraic curves and which has the same number of $k$-edges as $S$. Therefore, we can assume without loss of generality that $S$ is in general position with respect to degree $O(r)$ algebraic curves. This means that the number of points contained in $Z(f)$ is $O(r^2)$ and so the number of $k$-edges formed by two points both of which are in $Z(f)$ is $O(r^4)$. 
Finally, it is not hard to show that the number of $k$-edges formed by two points where one point is in $Z(f)$ and the other is not and the interior of the $k$-edge does not intersect $Z(f)$ is $O(n)$.\details{the only way this can happen if one point is in $Z(f)$ and the other is in one of the two cells which have the part of $Z(f)$ that contains the first point on their boundary. 
So there is a $O(r^2) \cdot 2 \cdot n/r^2$} 
Moreover, the number of $k$-edges formed by two points where one point is in $Z(f)$ and the other is not and the interior of the $k$-edge \emph{does} intersect $Z(f)$ is $O(nr)$, again by the $O(nr)$ bound on \cref{question}. 
Choosing $r = \Theta(n^{1/3})$ shows that the total number of $k$-edges is $O(n^{4/3})$. 
See \cref{sec:conclusion} for a discussion of why we believe the quantity in \cref{question} may be $\Theta(nr)$.

There are two technical issues introduced by our application of the polynomial partitioning theorem (\cref{thm:non-singularpartition}) to the probabilistic version of the $k$-edge problem. First, \cref{thm:non-singularpartition} can only be applied to distributions which have a density. The reason for this is that there is no known version of \cref{thm:non-singularpartition} which applies to distributions which do not have a density.\footnote{The proof of \cref{thm:non-singularpartition} relies on the Stone-Tukey ham sandwich theorem \cite{StoneTukey} for $L^1$ functions on $\Rl^d$. There is a version of the ham sandwich theorem which applies to more general distributions but it has a weaker conclusion and cannot be used to extend \cref{thm:non-singularpartition} to more general distributions as far as the authors know.}
Secondly, because we partition the plane with an arbitrary algebraic curve, in order to bound the number of times the $k$-edge graph intersects the boundary of the partition, we must use the convex and concave chain decompositions of the $k$-edge graph simultaneously. Because of this, the bound we prove in this section only matches the bound in \cref{thm:5/4} in the case when $k= \lfloor cn \rfloor $ for some $c \in (0,1)$.

Because of these technical issues, the theorem we prove in this section is weaker than \cref{thm:5/4}:
\begin{theorem}\label{thm:5/4poly}
Let $P$ be a probability distribution on $\Rl^2$ which has a density.
Then $E_P(k,n) \le  \frac{58n^{7/4}}{(k+1)^{1/4}(n-2-k)^{1/4}}$.
\end{theorem}

Before proving \cref{thm:5/4poly}, in \cref{sec:partition} we review the continuous polynomial partitioning theorem. 
\cref{sec:curves} establishes some necessary lemmas concerning algebraic curves.

We make one remark on the requirement that the probability distribution $P$ in \cref{thm:5/4poly} has a density. 
As mentioned earlier, in \cite{BaranySteiger}, B\'{a}r\'{a}ny and Steiger construct a probability distribution $P$ with $E_P(\frac{n-2}{2},n) = \Omega(n\log n)$. 
The distribution $P$ does not have a density. 
However, if $m_i$ is any decreasing sequence whose limit is zero, B\'{a}r\'{a}ny and Steiger also show how to construct a distribution $P'$ which has a density and with $E_{P'}(\frac{n-2}{2},n) = \Omega(m_nn\log n)$.
In particular, $E_{P'}(\frac{n-2}{2},n)$ can still be super-linear even if $P'$ has a density. 
This shows that the class of distributions which have a density is an important class of distributions to investigate in the context of the probabilistic $k$-facet problem. 
\details{Although the details haven't been worked out as far as we know, it is also probably not hard to construct a distribution $P$ which has a density and with $E_P(k,n) = \Omega(m_nne^{\Omega(\sqrt{\log n})})$}

\subsection{Polynomial partitioning}\label{sec:partition}
The polynomial partitioning theorem of \cite{GuthKatz} has recently been used to solve a number of problems in discrete and combinatorial geometry \cite{Guthbook}. It has also been used to give alternative proofs of some known results, see \cite{MR2957631}. Perhaps the most commonly used version of the polynomial partitioning theorem is the following theorem, which we refer to as the discrete version. 

\begin{theorem}[Discrete polynomial partitioning \cite{GuthKatz}]\label{thm:discretepartition}
Let $S\subset \Rl^d$ be a set of $n$ points. Then for each $r \le n$ there is a non-zero polynomial $f \in \Rl[x_1, \dotsc, x_d]$ of degree $O(r)$ (where the constants in big-$O$ depend only on $d$) such that $\Rl^d \setminus Z(f)$ is the union of a family $\mathcal{O}$ of $r^d$ pairwise disjoint open sets such that each $O\in \mathcal{O}$ contains at most $n/r^d$ points of $S$.
\end{theorem}

We are focused on bounding the expected number of $k$-edges of a sample of points from a distribution on the plane. Therefore, we need a partitioning theorem which applies to probability distributions rather than finite point sets. This type of result, which we refer to as the continuous version of the polynomial partitioning theorem, has been used to establish improved bounds for the restriction problem in harmonic analysis, see \cite{Guth} as a starting point. 

\begin{theorem}[Continuous polynomial partitioning \cite{Guth}]\label{thm:partition}
Let $W \in L^1(\Rl^d)$ with $W \ge 0$. Then for each $r$, there is a non-zero polynomial $f \in \Rl[x_1, \dotsc, x_d]$ of degree at most $r$ such that $\Rl^d \setminus Z(f)$ is the union of a family $\mathcal{O}$ of $\Theta(r^d)$ (where the constants in big-$\Theta$ depend only on $d$) pairwise disjoint open sets such that for all $O \in \mathcal{O}$, the integrals $\int_{O} W$ are equal. 
\end{theorem}
The open sets $O$ in \cref{thm:partition,thm:discretepartition} are called the \emph{cells} of the partition. 

Using the density of the non-singular polynomials in the space of all polynomials of fixed degree in $d$ variables, it is possible to obtain, as a corollary of the continuous polynomial partitioning theorem, a version where all the irreducible components of the dividing surface $Z(f)$ are non-singular varieties. 

\begin{theorem}[Non-singular continuous polynomial partitioning \cite{Guth}]\label{thm:non-singularpartition}
Let $W \in L^1(\Rl^d)$ with $W \ge 0$. Then for each $r$, there is a non-zero polynomial $f \in \Rl[x_1, \dotsc, x_d]$ of degree at most $r$ such that $\Rl^d \setminus Z(f)$ is the union of a family $\mathcal{O}$ of $\Theta(r^d)$ (where the constants in big-$\Theta$ depend only on $d$) pairwise disjoint open sets such that for all $O \in \mathcal{O}$, the integrals $\int_{O} W$ are within a factor of 2 of each other. 
Furthermore, all irreducible components of $Z(f)$ are non-singular.  
\end{theorem}

%As far as the authors know, it is an open problem to prove a polynomial partitioning theorem similar to \cref{thm:partition} for arbitrary probability distributions. 
%The proof of \cref{thm:partition} relies on the Stone-Tukey ham sandwich theorem \cite{StoneTukey} for $L^1$ functions on $\Rl^d$.
%In order to extend the polynomial partitioning theorem to arbitrary measures, one could attempt to use the more general form of the ham sandwich theorem and repeat the proof of \cref{thm:partition}. 
%However, since an arbitrary distribution may be concentrated in a hyperplane or even a point, one would not be able to guarantee that the integrals $\int_{O_i} W$ are all equal, but only that all integrals are at most some fixed value.\lnote{this paragraph is a bit confusing. Is it possible or impossible or unknown?}

\subsection{Counting intersection points of an algebraic curve and the \texorpdfstring{$k$}{k}-edge graph}\label{sec:curves}

Our use of the polynomial partitioning technique requires us to bound the number of times the $k$-edge graph $G_k$ of a set of $n$ points can intersect a degree $r$ algebraic curve $Z(f)$, i.e., we must give some answer to \cref{question}. Note that the number of points of intersection of $G_k$ and $Z(f)$ could be infinite if $Z(f)$ contained one of the lines spanned by a $k$-edge in $G_k$. However, for our purposes, it suffices to bound the number of intersection points in the case when it is finite. 

In order to establish our $O(nr^2)$ bound on the quantity in \cref{question}, we first show how to partition an irreducible algebraic curve into the union of $O(r^2)$ convex and concave pieces (\cref{irrdecomp}). Combining the convex/concave chains decomposition of the $k$-edge graph $G_k$ with the partition of a degree $r$ algebraic curve into $O(r^2)$ convex and concave pieces allows us to show that a degree $r$ algebraic curve intersects the $k$-edge graph of a set of $n$ points at most $O(nr^2)$ times assuming the number of intersections is finite (\cref{lem:intersections}). 

First, we show how to partition an irreducible curve $Z(f)$ into the union of a finite number of points and a finite number of convex/concave $x$-monotone connected curves. 

\begin{definition}
A connected curve $C\subset \Rl^2$ is \emph{$x$-monotone} if every vertical line intersects it in at most one point. 
\end{definition}

\begin{definition}
An $x$-monotone curve $C$ is \emph{convex (respectively, concave)} if for every three points $(x_1,y_1),(x_2,y_2) , (x_3,y_3)\in C$ with $x_1<x_2<x_3$, the point $(x_2,y_2) $ is below (respectively, above) or on the line joining $(x_1,y_1)$ and $(x_3,y_3)$. 
\end{definition}

In order to break $Z(f)$ into convex/concave pieces, we need to use the inflection points of $Z(f)$. 

\begin{definition}[\cite{MR1159092}]
A non-singular point $(a,b)$ of an algebraic curve $Z(f)$ is an \emph{inflection point} if the \emph{Hessian curve} $H_f(x,y):=f_y^2f_{xx}-2f_xf_yf_{xy}+f_x^2f_{yy}$ is equal to zero at $(a,b)$. (The notation $f_x$ denotes the partial derivative with respect to $x$.)
\end{definition}
\begin{proposition}\label{irrdecomp}
A non-singular irreducible curve $Z(f) \subset \Rl^2$ of degree $r$ that is not a vertical line can be partitioned into the union of at most $4r^2$ points and at most $6r^2$ $x$-monotone curves where each $x$-monotone curve is either convex or concave.
\end{proposition}
\begin{proof}
If $Z(f)$ is a non-vertical line, the conclusion is clearly true. So assume that $Z(f)$ is not a line. 
Let $F=Z(f) \cap Z(f_y)$. 
We know that $f$ depends on $y$ and not just $x$ because otherwise $Z(f)$ would be a vertical line. This means that $f_y$ is not identically zero.
Now because $f$ is irreducible and the degree of $f_y$ is less than the degree of $f$, the  polynomials $f$ and $f_y$ cannot have a common factor.
Therefore, by B\'ezout’s theorem, $|F| \le r(r-1)$.

Let $I$ be the set of inflection points of $Z(f)$.
An irreducible curve of degree $r\ge 2$ has at most $3r(r-2)$ inflection points \cite[Proposition 3.33]{MR1159092} so $|I| \le 3r(r-2)$.\details{real affine inflection is subset of complex projective inflection.} 
Let $\mathcal{C}$ be the set of connected components of $Z(f) \setminus (I \cup F)$. 
Because of the removal of the points in $F$, every curve in $\mathcal{C}$ is $x$-monotone.\details{Let $C \in \mathcal{C}$. 
Assume there exists two points $a,b  \in C$ that have the same $x$-coordinate. 
Since $C$ is not a vertical line, there must be a point $x \in C$ that is between $a,b$ and such that $x$ is not on the line through $a,b$. 
Therefore, between $a$ and $b$, the curve must travel in the positive $x$ direction and then in the negative $x$-direction, meaning there exists a point on the curve between $a$ and $b$ where $f_y=0$, a contradiction.} 

Now we show that, because of the removal of the inflection points $I$, every curve in $\mathcal{C}$ is either a convex $x$-monotone curve or a concave $x$-monotone curve. 
Each curve in $\mathcal{C}$ is the graph of a function defined in an interval. 
%Since $Z(f)$ is non-singular, the second derivative of this function is defined for every point in $Z(f)$.\lnote{something odd here: you got this directly, but later you had to invoke the implicit function theorem. Why the discrepancy?} 
We claim that, for each curve in $\mathcal{C}$, the second derivative of the associated function exists everywhere and is never zero. 
Let $C \in \mathcal{C}$. 
%Since $f_y \neq 0$ for all points in $C$, by the implicit function theorem, $C$ can be defined locally as the graph of a smooth function \cite[Remark 3.31]{MR1159092}. Furthermore, the second derivative of this function vanishes at a point $a$ with $(a,b) \in C$ if and only if $(a,b)$ is a point of inflection of $Z(f)$ \cite[Remark 3.31]{MR1159092}. Therefore, when $C$ is regarded as the graph of a function on an interval, the second derivative of this function is never zero. This means that $C$ is the graph of a convex or concave function and so $C$ is a convex or concave $x$-monotone curve.
Since $C$ does not contain a point where $f_y=0$, using the implicit function theorem, for each $(u,v) \in C$, there exists a smooth function $\phi: (u-\epsilon,u+\epsilon) \to \Rl$ that gives a local parameterization of the curve near $(u,v)$ \cite[Theorem 4.22]{Rutter}. 
Now a simple calculation shows that if $\phi''(x)$ is equal to 0 at $x$, then the Hessian curve $f_y^2f_{xx}-2f_xf_yf_{xy} + f_x^2f_{yy}$ equals zero at $(x,\phi(x))$.\details{We have parameterized $Z(f)$ near a given point in $C$ by $x \mapsto (x,\phi(x))$. 
Differentiating $f(x,\phi(x))=0$ gives $f_x+ \phi'(x)f_y = 0$ and so $\phi'(x) = -\frac{f_x}{f_y}$. 
Differentiating again gives $\phi''(x)f_y + (1,\phi'(x))  \begin{pmatrix} f_{xx} & f_{xy}\\ f_{xy} & f_{yy}\end{pmatrix} \begin{pmatrix}1 \\ \phi'(x) \end{pmatrix} = 0$ and now rearranging and using the fact that $\phi'(x) = -\frac{f_x}{f_y}$ and the fact that $f_y\neq 0$  shows that if $\phi''(x)=0$  then the Hessian curve is zero.}
The inflection points of $Z(f)$ are precisely the points where the Hessian curve is zero. 
Therefore, since all inflection points were removed, the second derivative of the function whose graph is $C$ is never zero. 
This means that the function is either strictly convex of strictly concave, and so $C$ is either a convex or concave $x$-monotone curve.

Now we determine how many distinct curves $\mathcal{C}$ can contain. The number of connected components of $Z(f)$ is at most $2r^2$ by either \cite[Theorem 2]{Milnor} or \cite[Theorem 2.7]{Sheffer}.
We removed at most $r(r-1)+ 3r(r-2)$ points from $Z(f)$. 
Because $Z(f)$ is non-singular, it has no points of self-intersection. 
Therefore, each point which is removed increases the number of connected components of $Z(f) \setminus (I \cup F)$ by at most 1. 
Therefore, the number of connected components of $Z(f) \setminus (I \cup F)$ is at most $2r^2 + r(r-1)+3r(r-2) \le  6r^2$. 
\iffalse
It suffices to break the curve at each inflection point and each singular point. An irreducible curve of degree $r$ has at most $3r(r-2)$ inflection points \cite[Proposition 3.33]{MR1159092} and at most $\frac{(r-1)(r-2)}{2}$ singular points \cite[Section 5.4]{MR1042981}. We get that the total number of break points required is at most $3r(r-2)+\frac{(r-1)(r-2)}{2}<4r^2$.
\fi 
\end{proof}

The decomposition into convex/concave pieces is useful because of the following fact:

\begin{lemma}\label{lem:convexconcave}
Let $C$ be an $x$-monotone convex curve and $D$ an $x$-monotone concave curve. If the number of points of intersection of $C$ and $D$ is finite, then it is at most 2. 
\end{lemma}
\begin{proof}
Assume that $C$ and $D$ intersect in three points $(x_1,y_1),(x_2,y_2),(x_3,y_3)$. Observe that the three points $(x_1,y_1),(x_2,y_2),(x_3,y_3)$ must be contained in a line $\ell$ and we may assume that $x_1<x_2<x_3$. 
We claim that $C$ and $D$ must both contain the line segment connecting the three points.
Indeed, assume that $C$ does not contain the line segment connecting $(x_1,y_1)$ and $(x_2,y_2)$. 
Then there must be a point $(x_0,y_0) \in C$ with $x_1<x_0<x_2$ and $(x_0,y_0)$ strictly below the line connecting $(x_1,y_1)$ and $(x_2,y_2)$. 
But then the point $(x_2,y_2)$ is above the line connecting $(x_0,y_0)$ and $(x_3,y_3)$, a contradiction to convexity of $C$. 
The argument for the other cases is similar. 
\end{proof}

Now we can establish the bound on the number of intersection points between $Z(f)$ and $G_k$. 

\begin{lemma}\label{lem:intersections}
Let $S \subset \Rl^2$ be a set of points in general position, $G_k = (S,E_k)$ the $k$-edge graph of $S$, and $f \in \Rl[x_1,x_2]$ a degree $r$ polynomial such that all irreducible components of $Z(f)$ are non-singular and $S \cap Z(f) = \emptyset$. If the number of points of intersection of $Z(f)$ and $G_k$ is finite, then it is at most $13nr^2$. 
\end{lemma}
\begin{proof}
First assume that $Z(f)$ is irreducible. If $Z(f)$ is a line, then it follows from \cref{lem:convexchains} that the number of intersection points is at most $\max(k+1,n-k-1)\le 13n$ and we are done. So assume that $Z(f)$ is not a line. First, we need to decompose the curve $Z(f)$ into the union of convex and concave pieces. By \cref{irrdecomp}, $Z(f)$ can be partitioned into the union of $6r^2$ convex/concave $x$-monotone curves and at most $4r^2$ points. Let $A$ be the set of convex $x$-monotone curves, $B$ the set of concave $x$-monotone curves, and $N$ the set of points in the partition. 

By \cref{lem:convexchains}, $G_k$ can be decomposed into the union of $k+1$ convex chains $C_1, \dotsc, C_{k-1}$, or $n-k-1$ concave chains $D_1, \dotsc, D_{n-k-1}$.

Recall that we are assuming that no line spanned by two points from $S$ is vertical. 
Therefore, the convex chains $C_1, \dotsc, C_{k+1}$, and the concave chains $D_1, \dotsc, D_{n-k-1}$ never contain two points on a vertical line. 
Furthermore, we claim that any convex or concave chain $C_i$ or $D_j$ intersects $Z(f)$ in only finitely many points. 
Indeed, if one of these chains intersected $Z(f)$ in infinitely many points, one of the line segments in the chain would have to intersect $Z(f)$ in infinitely many points. 
Recall the fact that if a degree $r$ algebraic curve intersects a line in more than $r$ points, then the curve must contain the line. 
Since we are assuming that $Z(f)$ does not contain any of the points in $S$, this is not possible.
Thus, we can apply \cref{lem:convexconcave} and the concave chain decomposition of $G_k$ to show that the number of $k$-edges in $G_k$ that intersect $Z(f)$ at a point contained in one of the convex $x$-monotone curves in $A$ is at most $2 (n-k-1) 6r^2$. 
Similarly, the number of $k$-edges in $G_k$ that intersect $Z(f)$ at a point contained in one of the concave $x$-monotone curves in $B$ is at most $2 (k+1) 6r^2$. 
Additionally, there are $4r^2$ points in the set $N \subset Z(f)$ which are not contained in any of the convex/concave $x$-monotone curves. 
Therefore, the total number of intersections is at most $2 (n-k-1) 6r^2+ 2 (k+1) 6r^2 +4r^2 \le 13nr^2$. 

If $Z(f)$ is not irreducible, then say $Z(f)$ is the union of $m$ irreducible components $f_1,f_2, \dotsc ,f_m$ of degrees $r_1,r_2, \dotsc, r_m$. 
By the above, the number of intersection points of $Z(f_i)$ and $G_k$ is at most $13nr_i^2$. 
So the total number of intersection points of $Z(f)$ and $G_k$ is at most $\sum_{i=1}^m 13nr_i^2  \le 13nr^2$.\details{because of additivity of degrees}
\end{proof}

\subsection{Proof of \texorpdfstring{\cref{thm:5/4poly}}{Theorem \ref{thm:5/4poly}}}\label{sec:pfpoly}

\begin{proof}
As in the proof of \cref{thm:5/4}, the conclusion is clearly true when $k=0$ so we can assume that $1 \le k \le (n-2)/2$.

Let $W$ be the density of $P$. For any fixed $n$, we use \cref{thm:non-singularpartition} applied to $W$ to find a degree $r:=r(n)$ (to be chosen later) polynomial $f$ which divides $\Rl^2 \setminus Z(f)$ into a family $\mathcal{O}$ of $\Theta(r^2)$ pairwise disjoint open sets such that for all $O \in \mathcal{O}$, the integrals $\int_{O}W $ are within a factor of 2 of each other.
Furthermore, all irreducible components of $Z(f)$ are non-singular. 

Let $X = \{X_1, \dotsc, X_n\}$ be a sample of $n$ points from $P$. Observe that the probability that two points in $X$ span a vertical line is 0 so we do not need to consider this case.
Also, since $P$ has a density, the measure of every line is 0. This means that $X$ is in general position with probability 1 so we may assume this as well. Therefore, we can analyze the $k$-edges of $X$ using the $k$-edge graph $G_k$ of $X$. Also, since the Lebesgue measure of $Z(f)$ is zero, $X \cap Z(f) = \emptyset$ with probability 1 so we assume this as well. 

As in the proof of \cref{thm:5/4}, we first bound the expected number of $k$-edges \emph{that are in $G_k$}.

We compute the expected number of $k$-edges in $G_k$ by considering two different types of $k$-edges separately. First we bound the expected number of $k$-edges formed by two points in different cells of the partition. Then we bound the expected number of $k$-edges formed by two points in the same cell of the partition. That is, the expected number of $k$-edges in $G_k$ is equal to
\begin{equation}\label{eq:decomp2}
\begin{split}
& \mathbb{E}(\text{number of } k\text{-edges in }G_k \text{ formed by two points in different cells}) \\ 
&+\mathbb{E}(\text{number of } k\text{-edges in }G_k \text{ formed by two points in the same cell}).
\end{split}
\end{equation}
If $\conv(X_1,X_2)$ is a $k$-edge formed by two points $X_1,X_2$ in different cells, then $\conv(X_1,X_2)$ intersects $Z(f)$. 
So to bound the expected number of $k$-edges in $G_k$ formed by two points in different cells, it suffices to bound the expected number of $k$-edges that intersect $Z(f)$.
We claim that the number of points of intersection between $G_k$ and $Z(f)$ is finite. This is true because otherwise some $k$-edge would have to intersect $Z(f)$ infinitely many times. If a degree $r$ algebraic curve intersects a line more than  $r$ times it must contain that line. If this were true, then $Z(f)$ would have to contain the two points of $X$ forming the line, but we are assuming that $X \cap Z(f) = \emptyset$. Therefore, the number of points of intersection between $G_k$ and $Z(f)$ is finite and so we can apply \cref{lem:intersections} to show that the first term in \cref{eq:decomp2} is at most $13nr^2$. 

Now we bound the second term. 
Recall that $Z(f)$ divides $\Rl^2 \setminus Z(f)$ into a family $\mathcal{O}$ of $\Theta(r^2)$ pairwise disjoint open sets (called cells) such that for all cells $O \in \mathcal{O}$, the integrals $\int_{O}W $ are within a factor of 2 of each other.
Therefore, for any fixed $i \neq j$, the probability that $X_i$ and $X_j$  are in the same cell is at most $4/r^2$.

Now, the second term in \cref{eq:decomp2} can be bounded using nearly the same argument which we used to bound the second term in \cref{eq:decomp} in the proof of \cref{thm:5/4}. This argument is given in \cref{eq:cond,eq:int} in the proof of \cref{thm:5/4}. The only change is that for any $X_i,X_j \in X$, the probability $\mathbb{P}(X_1 ,X_2 \text{ are in the same cell})$ is now upper bounded by $4/r^2$ instead of $1/m$. 
Therefore, we have that \cref{eq:decomp2} is at most
\[
13nr^2 + 4\frac{n^2}{r^2} \frac{\sqrt{n-2}}{\sqrt{k}\sqrt{n-2-k}}
\]
and choosing $r^2 = \Big\lfloor \frac{n^{3/4}}{(k+1)^{1/4}(n-2-k)^{1/4}} \Big\rfloor$ makes the above quantity less than \\ $\frac{29n^{7/4}}{(k+1)^{1/4}(n-2-k)^{1/4}}$.
Since we could repeat the argument after rotating the plane 180 degrees, the expected number of $k$-edges is at most two times the bound we just obtained for the expected number of $k$-edges in $G_k$. 
\end{proof}

\section{Tightness of the argument in \texorpdfstring{\cref{thm:basic}}{Theorem \ref{thm:basic}}: Translations of a fixed shape}\label{sec:tight}

In this section we study a natural variation of the $k$-set problem for translations of a fixed convex set on the plane, namely, the number of ways in which one can enclose $k$ points out of a given finite set of points by a translation of a convex set so that its boundary strictly separates them from the rest.
We will show nearly matching upper and lower bounds on the expected number of ways.

For our lower bound, one of our tools will be the uniform convergence theorem of Vapnik and Chervonenkis \cite{VapChe71}.
This introduces a minor technical complication: their theorem is about abstract set systems without regard to whether sets have a boundary, while the standard $k$-set problem for lines on the plane ask for strict separation by a line and therefore the natural choice for our generalization is to ask for strict separation by a curve.

Similarly, the other side of our argument, our upper bound, is a variation on the two-step argument in the proof of \cref{thm:basic} (\cref{eq:BaranySteiger} and the general upper bound of the integrand in \cref{eq:generalbound}), which uses $k$-edges and therefore also uses the boundary curve in a fundamental way.

A \emph{convex body} is a compact convex set with non-empty interior.
A set $C$ is \emph{strictly convex} if for all $x, y \in C$ with $x \neq y$ and for all $\lambda \in (0,1)$ we have $\lambda x + (1-\lambda) y \in \interior C$.
%For a set system $(X,R)$ and a finite subset $\pset \subseteq X$, an $R$-$k$-set of $\pset$ is a subset $T \subseteq \pset$ of $k$ points and of the form $T = \pset \cap Q$ for some $Q \in R$.
A \emph{set system} is a pair $(X,\mathcal{R})$, where $X$ is a set and $\mathcal{R}$ is a family of subsets of $X$.
The elements of $\mathcal{R}$ are called \emph{ranges}.
For a set $C \subseteq \RR^2$, let $(\RR^2,T_C)$ be the set system of translations of $C$ (that is, $T_C$ is the family of translation of $C$).
%It is conventional to call the sets in $T_C$ \emph{ranges}.
We are interested in translations of convex sets and it will be notationally convenient to set $C$ to be the \emph{interior} of a fixed convex body. 
So, for this section, $C$ will be restricted (at least) to be the interior of a convex body.
In this case, when we say that a point lies on the boundary of a range, the point does not lie in the range.
\begin{definition}\label{def:T_C}
For a finite subset $\pset \subseteq \RR^2$, a $T_C$-$k$-set of $\pset$ is a subset $T \subseteq \pset$ of $k$ points such that for some $Q \in T_C$, $\pset \cap \bd Q = \emptyset$ and $T = \pset \cap Q$.
\end{definition}
%For consistency with the standard $k$-set problem for lines on the plane (where we think of a line as the boundary of some open halfplane and points of a $k$ set have to lie in the \emph{open} halfplane), $C$ should be an \emph{open} set.
%Note that when we say that a point lies on the boundary of a range, the point may or may not lie in the range (depending on whether the range contains its own boundary).

\subsection{Upper bound for \texorpdfstring{$T_C$-$k$}{TC-k}-sets, probabilistic, \texorpdfstring{$k$}{k} proportional to \texorpdfstring{$n$}{n}}\label{sec:strictlyconvex}

This section establishes our upper bound on the expected number of $T_C$-$k$-sets of a set of $n$ iid points when $C$ is the interior a strictly convex body and $k$ is proportional to $n$.
So for \cref{sec:strictlyconvex}, let $C \subseteq \RR^2$ be the interior of a strictly convex body.

\begin{definition}
A set of points in $\RR^2$ is in \emph{general position relative to $C$} if no three points lie on the boundary of some range in $T_C$ (i.e., some translation of $C$).
\end{definition}

\begin{lemma}\label{lem:edges}
Let $(p,q)$ be a pair of distinct points in $\RR^2$.
Then there are at most two ranges $x+C$ satisfying $p, q \in \bd(x+C)$.
\end{lemma}
\begin{proof}
Up to a rotation we can assume that $r:=q-p$ is vertical.
Suppose for a contradiction that there are three ranges $x+C$ satisfying $p, q \in \bd(x+C)$.
This implies there are three points $p_1$, $p_2$, $p_3$ such that $p_1, p_2, p_3, p_1+r, p_2+r,p_3+r \in \bd C$.
Let $f(x_1)$ denote the length of segment ``$C$ intersected with the vertical line at $(x_1, 0) \in \RR^2$.''
Function $f$ is positive in a non-empty interval $(a,b)$, is strictly concave in $[a,b]$ and takes value $\norm{r}$ at three points in $[a,b]$.
But there is no such function so this is a contradiction.
\end{proof}
From the lemma we conclude: 
\begin{corollary}\label{cor:functionC}
Let $V \subseteq (\RR^2)^2$ be the set of pairs of distinct points that can appear on the boundary of some range. 
Then there exists a continuous onto function $\fC: V \to T_C$.
\end{corollary}
\begin{proof}
Every pair in $V$ can appear on the boundary of one or two ranges.
When a pair appears on the boundary of exactly one range, map both orderings of the pair to that range.
%old: When a pair appears on the boundary of two ranges, map one ordering of the pair of points into one range and the other ordering into the other range arbitrarily.
When a pair $(p,q)$ appears on the boundary of two ranges, let $C_1$ and $C_2$ be the two translations of $C$ containing $p$ and $q$ on their boundaries.
Let $\fC(p,q)$ be the unique translation that solves $\max_{i\in \{1, 2\}} \area(C_i \cap \aff(p,q)^+)$ (where $\aff(p,q)^+ = \{r \suchthat -(q-p)_x (r-p)_y + (q-p)_y (r-p)_x > 0 \}$).
%i.e. stand at p, face towards q, positive is to your right.
%\details{there seems to be a canonical way of implementing this arbitrary choice by using orientation}
\end{proof}
% For notational convenience we define $\fC(p,q) = \emptyset$ when $p,q$ cannot appear on the boundary of any range. 

% Fix an arbitrary choice of a function $\fC(\cdot, \cdot)$ given by \cref{cor:functionC} and define:
From now on we let $\fC(\cdot, \cdot)$ denote the function given by \cref{cor:functionC} (with a slight abuse of notation).
\begin{definition}
For a set of points $\pset$ in general position relative to $C$, let a (oriented) \emph{$T_C$-$k$-edge} be an ordered pair of points $(p,q) \in V$ with $p,q \in \pset$ ($p\neq q$) such that $\fC(p,q)$ contains $k$ points of $\pset$.\details{general position implies boundary contains only $p$, $q$}
\end{definition}

We now show a bound relating $T_C$-$k$-sets and $T_C$-$k$-edges, which follows from a variation of known continuous deformation arguments \cite[Lemma 5.15]{Sariel}, \cite[Chapter 11]{Matousek}.
For a finite set of points $\pset$ in general position relative to $C$, let $e_k(\pset)$ be the number of $T_C$-$k$-edges of $\pset$ and let $a_k(\pset)$ be the number of $T_C$-$k$-sets of $\pset$.
\begin{lemma}\label{lem:shp} %[based on SHP, and relationship between k edges and k sets]
For a finite set of points $\pset$ in general position relative to $C$ and $k \geq 2$ we have $a_k(\pset) \leq 4\bigl(e_{k-2}(\pset) + e_{k-1}(\pset) + e_{k}(\pset)\bigr)$.
\end{lemma}
\begin{proof}
%Let $\pset$ be the finite set of points. 
To prove the claim we will construct an injective function $f$ from $T_C$-$k$-sets of $\pset$ to a labelled extension of the set of $T_C$-$k$-edges.
The function is defined as follows: 
Let $Q \subseteq \pset$ be a $T_C$-$k$-set induced by some range $C_0$.
Translate $C_0$ in the $x$ direction until some point $p \in \pset$ lies on its boundary to obtain range $C'$, then translate $C'$ while keeping $p$ on its boundary (letting $p$ slide along the boundary) until another point $q \in \pset$ lies on the boundary to obtain a range $C''$ (there may be more than one choice here, pick arbitrarily).
%We have that $\cl C''$ contains $Q$ and between zero and two other points from $\pset$.
We have that $\cl C''$ contains $Q$ and at most two other points from $\pset$.
From the general position assumption, $\bd C''$ contains exactly two points from $S$. 
Swap points $p, q$ if needed so that $C'' = \fC(p,q)$.
Pick labels $l_p, l_q \in \{IN, OUT\}$ for $p$ and $q$ according to whether they are in $Q$.
This completes the definition of an $f$ from $T_C$-$k$-sets of $\pset$ to $\pset^2 \times \{IN, OUT\}^2$, namely 
$f(Q) = (p,q,l_p, l_q)$.

% Function $f$ is injective because it can be inverted in its image in the following way:
% For a pair of points $(p,q)$ and labels $(l_p, l_q)$, let $f^{-1}(p,q,l_p,l_q)$ be $\fC(p,q) \cap \pset$ with $p$ and $q$ added depending on labels $l_p, l_q$, respectively.\details{$\interior \fC(p,q) \cap \pset$ union any subset of $\{p,q\}$ is a $T_C$-$k$-set because can shift $\fC(p,q)$ a bit to obtain.\lnote{?}\bnote{How do you know that there exists a translation of $C$ that contains $\fC(p,q)\cap S$ and also contains $p,q$ depending on labels? For example, what if $C$ is a circle and $p,q$ are two points which form a horizontal segment that is a diameter of $C$ and the labels indicate that both $p$ and $q$ should be added? This can't happen but does it require proof?}}

We now show that $f$ is injective.
Let $Q$, $Q'$ be two $T_C$-$k$-sets of $\pset$ induced by ranges $C_0$, $C_0'$, respectively, and so that $f(Q)=f(Q')=(p,q,l_p,l_q)$.
By definition of $f$ we have that $Q$ is equal to $\pset \cap \fC(p,q)$ with $p$ and $q$ added according to the labels.
But then by definition of $f$ we have that $Q'$ is also equal to that set and therefore equal to $Q$.
This establishes that $f$ is injective.

To conclude, the image of $f$ contains only pairs $(p,q)$ that are $T_C$-$r$-edges of $\pset$ for $r \in \{k-2, k-1, k\}$. 
The claim follows.
\end{proof}

\begin{assumption}\label{assumption:mu}
Given $C$, probability distribution $\pdist$ on $\RR^2$ is such that $\pdist\bigl(\bd(x+C)\bigr)=0$ for all $x\in \RR^2$.
\end{assumption}
(In particular the assumption on $\pdist$ holds if $\pdist$ has a density.)

\begin{proposition}\label{prop:distinct}
Let $\pdist$ be a Borel probability distribution satisfying \cref{assumption:mu}.
Let $Y,Z$ be a pair of iid points, each according to $\pdist$.
Then $Y\neq Z$ a.s.
\end{proposition}
\begin{proof}
Fix a point $b$ on the boundary of $C$ (so that the origin is on the boundary of $-b+C$).
Note that $\pr(Y = Z) = \e\bigl( \pr(Y = Z \giventhat Z )\bigr) \leq \e\Bigl( \pr \bigl(Y \in \bd(Z-b+C) \bigm| Z \bigr) \Bigr) = 0$.
\end{proof}

\begin{proposition}\label{prop:gp}
Let $\pdist$ be a Borel probability distribution satisfying \cref{assumption:mu}.
Let $X$ be a random set of $n$ iid points, each according to $\pdist$.
Then $X$ is in general position relative to $C$ a.s.
\end{proposition}
\begin{proof}
It is enough to prove the claim for $n=3$.
Let $Y,Z,W$ be three iid random points according to $\pdist$.
By \cref{prop:distinct}, $Y \neq Z$ a.s.
Then 
\begin{align*}
&\pr \bigl((\exists a)Y,Z,W \in \bd(a+C) \bigr)\\
&=  \pr \bigl(Y\neq Z, (\exists a)Y,Z,W \in \bd(a+C) \bigr) \\
&= \e\Bigl( \pr \bigl( (\exists a)Y,Z,W \in \bd(a+C) \bigm| Y,Z \bigr) \Bigm| Y \neq Z \Bigr) \\
&= \e\Bigl( \pr \bigr( W \in \bd \fC(Y,Z)\text{ or }  W \in \bd \fC(Z,Y) \bigm| Y,Z \bigl) \Bigm| Y \neq Z\Bigr) \\ 
&= 0.\qedhere
\end{align*}
\end{proof}

\begin{theorem}[$T_C$-$k$-set/edge upper bound, probabilistic, $k$ proportional to $n$]\label{thm:halvingaverage}
Let $c \in (0,1)$.
Let $\pdist$ be a Borel probability distribution satisfying \cref{assumption:mu}.
Let $X$ be a random set of $n$ iid points, each according to $\pdist$.
Let $A_n$ (resp. $E_n$) be the number of $T_C$-$k$-sets (resp. $T_C$-$k$-edges) of $X$ for $k =\floor{cn}$.
Then
\[
\e(E_n) \leq O(n^{3/2})
\]
and
\[
\e(A_n) \leq O(n^{3/2})
\]
(where the constants in big-O depend only on $c$).
\end{theorem}
\begin{proof}
Let $\fC(p,q)$ and $V$ be as in \cref{cor:functionC}.
From \cref{prop:gp}, $X$ is in general position relative to $C$ a.s.

%\lnote{remaining issues: assumption on mu implies general position, what if \fC(x,y) is not defined for some x,y?} 
%Let $E_n$ be the number of $T_C$-$k$-edges of a random set of $n$ iid points according to distribution $\pdist$.
% We have $\pr(\text{$(X_i, X_j)$ is a $T_C$-$k$-edge})$ is 
%Let $ $ be two iid points according to $\pdist$.
Let $X = \{ X_1, \dotsc, X_n \}$.
Let $T =\pdist\bigl(\fC(X_1, X_2)\bigr)$ with the additional convention that $\fC(p,q)=\emptyset$ if $(p,q) \notin V$.
In this way $T$ is defined whenever $X_1 \neq X_2$, that is, a.s.\ by \cref{prop:distinct}.
Let $G_P(t) = \pr (T \leq t)$ for $t \in \RR$. 
Using a variation of \cref{eq:BaranySteiger} and the argument in the proof of \cref{thm:5/4} we get:
\begin{align*}
&\pr \bigl( \text{$(X_1,X_2)$ is a $T_C$-$k$-edge of $X$} \bigr) \\
&= \pr \bigl(\card{\fC(X_1, X_2) \cap (X \setminus \{X_1, X_2\})} = k \bigr) \\
&= \e \Bigl(\pr \bigl( \card{\fC(X_1, X_2) \cap (X \setminus \{X_1, X_2\})} = k \bigm| X_1, X_2 \bigr) \Bigr) \\
&= \binom{n-2}{k} \e \bigl(T^k (1-T)^{n-2-k} \bigr) \\
&= \binom{n-2}{k} \int_0^1 t^k (1-t)^{n-2-k} \ud G_P(t)
\end{align*}
and
\begin{align*}
\e(E_n) &= n(n-1) \binom{n-2}{k} \int_0^1 t^k (1-t)^{n-2-k} \ud G_P(t) \\
%\leq 2 n^2 \binom{n-2}{\floor{(n-2)/2}}\frac{1}{2^{n-2}}
&\leq n^2 \binom{n-2}{k} \bigg( \frac{k}{n-2} \bigg)^k  \bigg(\frac{n-2-k}{n-2} \bigg)^{n-2-k} \\
&\leq n^2 \frac{\sqrt{n-2}}{\sqrt{k}\sqrt{n-2-k}} \\
&\leq O(n^{3/2}).
\end{align*}
This proves the first inequality.
From this, the second inequality is immediate using \cref{lem:shp}.
\end{proof}

% \begin{theorem}[halving set upper bound for random]
% Let $A_n$ be the number of $T_C$-$k$-sets for $k=\floor{(n-2)/2}$ of a random set of $n$ iid points, each according to probability distribution $\pdist$ on $\RR^2$.
% Then
% \[
% \a(A_n) \leq O(n^{3/2}).
% \]
% (where the constants in big-O depend only on $\pdist$ and $C$).
% \end{theorem}
% \begin{proof}
% Immediate from \cref{lem:shp,thm:halvingaverage} 
% \end{proof}

% \begin{theorem}[corollary of lower bound via vc]
% Let $D$ be the uniform distribution in a convex body containing a translation of $C$.
% Then there exists a function $k(n)$ such that $e(n,k(n)) \geq \Omega(n^{3/2}/\sqrt{\log n})$ 
% \end{theorem}

\subsection{Lower bound for \texorpdfstring{$T_C$-$k$}{TC-k}-sets, deterministic, \texorpdfstring{$k$ proportional to $n$}{k proportional to n}}

In this section we show an $\Omega(n^2)$ lower bound on the maximum number of $T_C$-$k$-sets of a broad family of set systems of the form $(\RR^2, T_C)$ for $k$ proportional to $n$.
We illustrate the main idea of the argument in \cref{prop:square} for the case where $C$ is a unit square.
In \cref{thm:c2cross}, we then use the argument for the case where $C$ is the interior of a convex body with $C^2$ boundary (actually, slightly more general than that).

While the sets of points in the following results may not be in general position, this is not a true weakness of the results.
The reason is that, like in the case of the standard $k$-set problem for lines, the number of $T_C$-$k$-sets of a given set of points cannot decrease by applying any sufficiently small perturbation to the points, because any range inducing a $T_C$-$k$-set must by definition contain no point on its boundary.
Thus, the maximum number of $T_C$-$k$-sets among set in general position is no smaller than the number of $T_C$-$k$-sets among unrestricted sets of points.

To build intuition on an $\Omega(n^2)$ lower bound, we first consider in \cref{prop:square} the case of translations of a square and a set of $n$ equally spaced points arranged  in the shape of a cross (or plus sign, $+$).

\begin{proposition}\label{prop:square}
Let $C$ be the open unit square. 
Let $0<c<c'<1$.
% Then for $k=\Theta(n)$ the number of $T_C$-$k$-sets is $\Omega(n^2)$.
Then for $cn \leq k \leq c'n$ we have $\max_{\card{\pset}=n} a_k(\pset) = \Omega(n^2)$
(where the constants in $\Omega$ depend only on $c$ and $c'$).
\end{proposition}
\begin{proof}
We show the case where $n$ is a multiple of 4 and $k=n/2 - 2$, the rest is similar.
Consider a set of points equally spaced on the $x$ and $y$ axes forming a cross.
Say, for $\lambda = 4/n$, let $\pset = \lambda \bigl( \ZZ^2 \cap (\text{$x$-axis} \cup \text{$y$-axis}) \cap [-n/4,n/4]^2 \setminus (0,0) \bigr)$.
Then $\card{\pset}=n$ and the claim follows.
\end{proof}

For our next $\Omega(n^2)$ lower bound, \cref{thm:c2cross}, we assume that the boundary of $C$ is well approximated by its unique tangent line at certain points (locally of class $C^2$). 
See \cite[Section 5.1, subsection ``Second-Order Differentiability'']{MR2335496} for definitions and basic facts about differentiability of the boundary of a convex body.
The sets of points are adaptations of the ``cross'' construction from \cref{prop:square}.

\begin{definition}\label{def:C2}
Let $C \subseteq \RR^2$ be the interior of a convex body.
We say that $\bd C$ is $C^2$ at a point $x \in \bd C$ if the following conditions hold:
\begin{enumerate}
\item There is a unique support line $L$ of $C$ at $x$.
\item Apply an invertible affine transformation so that $x$ is at the origin, $L$ is horizontal and $C$ is above $L$. Then the lower side of $\bd C$ is represented in the form $(t,g(t))$ in a neighborhood of $0$ for a convex function $g$ such that there is a function $r(t)=at^2$ with $a>0$ and $g(t) = r(t) + o(t^2)$ as $t \to 0$.
\end{enumerate}
\end{definition}

\begin{theorem}\label{thm:c2cross}
%Assume additionally that $C$ has $C^2$ boundary. 
Assume that $C \subseteq \RR^2$ is the interior of a convex body such that
there exist linearly independent unit vectors $u,v \in \RR^2$ and points $a,b,c,d \in \bd C$ such that $\bd C$ is $C^2$ in a neighborhood of $a,b,c,d$ with outer normals $u,v,-u,-v$, resp.
Let $0<c<c'<1$.
Then for $cn \leq k \leq c'n$ we have $\max_{\card{\pset}=n} a_k(\pset) = \Omega(n^2)$
(where the constants in $\Omega$ depend only on $c$, $c'$ and $C$).
%\lnote{maybe can prove for C $C^2$ and convex but not necessarily strictly convex}
\end{theorem}
\begin{proof}
We show the case where $n$ is a multiple of 8 and $k=n/2$, the rest follows easily.
Up to an invertible linear transformation, we can assume $u=e_1$ and $v=e_2$, without loss of generality.
%For a non-zero vector $p$, let $v(p)$ be the unique point on the boundary of $C$ having outer normal $p$.
Let $U = \{e_1, e_2, -e_1, -e_2 \}$ and for $p \in U$ let $v(p)$ be a locally $C^2$ point on the boundary of $C$ and having outer normal $p$.

%Let $a_{1+}, a_{2+}, a_{1-}, a_{2-}$ be the four points on the boundary of $C$ having outer unit normal vectors $e_1, e_2, -e_1, -e_2$, respectively.
For $p \in U$ and some $t>0$, consider the segment of length $2t$ perpendicular to the boundary of $C$ at $v(p)$ and centered at $v(p)$, namely $s(p) := \conv\{v(p)-tp, v(p)+tp\}$.
Finally, consider a (one-dimensional) grid $g(p)$ of $n/4$ equally spaced points on segment $s(p)$.
Let our set of points be $\pset = \cup_{p \in U} g(p)$.
By construction $\card{\pset \cap C} = n/2$.
Let $\eps := 2t/(n/4-1)$ be the gap between consecutive points in each segment.

We will now show that we can choose $t>0$ small enough so that there are $\Omega(n^2)$ small translation of $C$ that induce different subsets of $\pset$, each containing $n/2$ points.
The idea of the argument is to translate $C$ independently in the vertical and horizontal direction, to pick $\Omega(n)$ different subsets of $n/4$ points among the pair of vertical segments and similarly for the horizontal segments.
The translations, notated $p+C$ and parameterized by $p$, form the following grid around the origin: $G := \{ (k\eps,l\eps) \in \RR^2 \suchthat k,l \in \ZZ \cap [-n/8,n/8]\}$.
By Taylor's theorem and compactness there exist constants $\alpha > 0, t_M>0$ (that depend only on $C$) such that the boundary of $C$ has a $C^2$ parametrization $y=\phi(x)$ in a neighborhood of $(x_0, y_0) := v(-e_2)$ such that $\abs{\phi(x)-y_0} \leq \alpha (x-x_0)^2$ for $x \in [x_0-t_M, x_0+t_M]$ (and similarly for $v(e_1), v(-e_1), v(e_2)$).
In particular, $\abs{\phi(x)-y_0} \leq \alpha t_M^2$.
We choose $t>0$ small enough so that $C$ contains the same subset of $g(-e_2)$ when translated distance less than or equal to $t$ in the horizontal direction.
Note that this also ensures that the boundaries of those translations contain no point of $g(-e_2)$.
The nearest point from $g(-e_2)$ to the line $y=y_0$ is at distance $\eps/2 = \frac{t}{n/4 - 1} > 4t/n$ so it is enough to have $\alpha t^2 \leq 4t/n$, that is, we set $t = \min\{t_M,\frac{4}{\alpha n}\}$.

With these choices, every $p \in G$ induces a different $T_C$-$k$-set of $\pset$ with $k=n/2$ and therefore $a_{n/2}(\pset) \geq \card{G} \geq n^2/16$.
\end{proof}

To understand the scope of \cref{thm:c2cross}, note that the condition on $C$ is satisfied when $C$ is the interior of a convex body with $C^2$ boundary.
Also, no triangle satisfies the assumptions of \cref{thm:c2cross}.

\subsection{Bounds on the growth function}

To put our results on the number of $T_C$-$k$-sets and $T_C$-$k$-edges in context, we state some basic bounds on the growth function of set system $(\RR^2,T_C)$, namely the maximum number of subsets of a set of $n$ point in $\RR^2$ induced by translations of $C$.
For simplicity some of our bounds have extra assumptions on $C$ that may not be necessary.

The \emph{growth function} \cite{VapChe71} of set system $(\RR^2,T_C)$ is given by 
\[ 
n \mapsto \max_{\pset \subseteq \RR^2, \card{\pset}=n} \bigl\lvert{\{\pset \cap (x+C) \suchthat x \in \RR^2\}}\bigr\rvert. 
\]

\begin{proposition}\label{cor:growth}
Let $C \subseteq \RR^2$ be the interior of a strictly convex body.
The growth function of $(\RR^2,T_C)$ is at most $n^2-n+2$.
\end{proposition}
\begin{proof}
For the proof we use the notions of a dual set system and dual growth function.
Given set system $(X, \mathcal{R})$, its \emph{dual set system} is $(\mathcal{R}, X^*)$, where $X^*$ is the family of sets of the form $\{ R \in \mathcal{R} \suchthat x \in R \}$ for $x \in X$.
The \emph{dual growth function} of $(X, \mathcal{R})$ is the growth function of its dual set system.

The first step is to notice that $(\RR^2,T_C)$ corresponds to the dual set system of $(\RR^2,T_{-C})$: 
a range $x-C \in T_{-C}$ is corresponds to point $x \in \RR^2$ and a point $y \in \RR^2$ corresponds to range $y+C$, together with the equivalence ``$y \in x-C$ is equivalent to $x \in y+C$''.
In this way, the growth function of $(\RR^2,T_C)$ is the dual growth function of $(\RR^2,T_{-C})$.

The second step is to bound the dual growth function of $(\RR^2,T_{-C})$.
Its value at $n$ is equal to the maximum number of nonempty cells in the Venn diagram of $n$ translations of $-C$.
It is bounded by the maximum number of connected components of the complement of $n$ translations of $\bd(-C)$, or, equivalently, $\bd(C)$.
Adding $n$ translations of $\bd(C)$ one by one, this number of connected components is 2 for $n=1$ and, using the fact that two translations of $\bd(C)$ intersect in at most two points\details{add proof sometime?}, it increases by at most $2(k-1)$ when the $k$th translation is added.
Therefore the number of connected components is at most $n^2-n+2$.
\end{proof}

For clarity we state the following summarizing result:
\begin{theorem}\label{thm:strictlygrowth}
Let $C \subseteq \RR^2$ be the interior of a strictly convex body with $C^2$ boundary.
The growth function of $(\RR^2,T_C)$ is $\Theta(n^2)$ (where the constants in $\Theta$ depend only on $C$).
\end{theorem}
\begin{proof}
Immediate from \cref{cor:growth,thm:c2cross}.
\end{proof}

In order to prove our results in \cref{sec:lowerbound} in more generality, we state here a weaker bound on the growth function with weaker assumptions on $C$.
\begin{theorem}\label{thm:growth}
Let $C \subseteq \RR^2$ be the interior of a convex body.
The VC-dimension of $(\RR^2,T_C)$ is at most $3$.
The growth function of $(\RR^2,T_C)$ is at most $\binom{n}{0}+\binom{n}{1}+\binom{n}{2}+\binom{n}{3} \leq (en/3)^3$.
\end{theorem}
\begin{proof}
%The VC-dimension bound is a special case of a result in \cite{NASZODI201077}.
A special case of a result in \cite{NASZODI201077} establishes that the VC-dimension of $(\RR^2,T_C)$ is at most $3$ when $C \subseteq \RR^2$ is a convex body.
The bound extends to our case (interior of a convex body) using the observation that if translations of the interior a convex body $C$ shatter a given finite set of points then translations of a scaled down $\cl C$ also shatter the same set.
The rest follows from the Sauer-Shelah lemma.
\end{proof}
The VC-dimension bound is tight: $C$ equal to any fixed triangle is a tight example.

\subsection{Lower bound for \texorpdfstring{$T_C$-$k$}{TC-k}-sets, probabilistic, some \texorpdfstring{$k$ proportional to $n$}{k proportional to n}}\label{sec:lowerbound}

In this section we show, for some $k$ proportional to $n$, an $\tilde\Omega(n^{3/2})$ lower bound for the expected number of $T_C$-$k$-sets for a random sample of $n$ points from the uniform distribution in a set $A \subseteq \RR^2$ sufficiently large to contain translations of $C$.
The restriction to a subset $A$ is necessary as there is no uniform distribution in $\RR^2$.
Our argument uses crucially the fact that translations of $C$ that are contained in $A$ have the same probability under the uniform distribution in $A$.
The minor technical complications introduced by the fact that $A$ is bounded could be avoided by considering a similar set system of translations of a disk (say) on the surface of the two-dimensional sphere (or translations of a shape on the flat torus) with the uniform distribution (a version not studied in this paper).

% We say that a convex body is in John position if its John ellipsoid is the unit ball centered at the origin
The idea of the proof is the following: 
First show that for a random sample $X$ of $n$ points in $A$, with high probability the number of induced subsets by translations of $C$ contained in $A$ is $\tilde\Omega(n^2)$ (\cref{lem:lowerboundrandom}).
Then, by VC's uniform convergence theorem, with high probability each translation of $C$ contained in $A$ contains $cn \pm \tilde O(\sqrt{n})$ points from $X$ for some $c$.
Therefore, by the pigeonhole principle there are $\tilde\Omega(n^{3/2})$ induced subsets of $X$ that contain exactly the same number of points, that is, $X$ has $\tilde\Omega(n^{3/2})$ $T_C$-$k$-sets for some $k$.

% We start by showing that if two translations of $C$ are close then they have a large intersection.
We start by showing that if two translations of $C$ are far apart then they have a small intersection.
\begin{lemma}\label{lem:distance}
%Let $C \subseteq \RR^2$ be a convex body in John position.
Let $C \subseteq \RR^2$ be the interior of a convex body that contains a unit ball.
If $\norm{x} \leq 1$, then 
\[
    \area \bigl(C \cap (x+C)\bigr) \leq (1-\norm{x}/2) \area(C).
\]
\end{lemma}
\begin{proof}
Consider the function $f(x) = \area \bigl(C \cap (x+C)\bigr)$.
It is logconcave (by the Pr\'ekopa-Leindler inequality and the fact that $f(x) = \one_C(x) * \one_{-C}(x)$) (where $\one_C(x)$ is the indicator function of $C$, namely one if $x \in C$ and zero otherwise).
Also, $f(0) = \area(C) \geq \pi$.
In other words, $\log f(x)$ is concave and, while it is not differentiable at $x=0$, we can use directional derivatives and tangent rays at $x=0$ to upper bound it by a function of the form $x \mapsto \log f(0) + c\norm{x}$, where $c<0$ is an upper bound on the one-sided directional derivative.

We calculate a suitable $c$ now.
The one-sided directional derivative at $0$ along unit vector $v \in \RR^2$ is 
\[
    D f(0)(v) = -2 \length(\text{projection of $C$ onto line perpendicular to $v$}) \leq -4
\] 
(from the analysis of the movement of chords of $C$ parallel to $v$\details{as a chord moves by distance $\Delta t$, it contributes $2\Delta t$ and the family of chords is parameterized by values in projection of $C$ onto line perpendicular to $v$}).
Thus, $\log f(0) = \log(\area C)$ and $D (\log f)(0) (v) = Df(0)(v)/f(0) \leq -4/\pi \leq -1$ (i.e. we can take $c=-1$) and these estimates with concavity of $\log f(x)$ give $\log f(x) \leq \log \bigl(\area(C)\bigr) - \norm{x}$.
That is, $f(x) \leq \area(C) e^{-\norm{x}}$.
We use the inequality $e^{-t} \leq 1-t(1-1/e)$ for $t \in [0,1]$ to conclude that if $\norm{x}\leq 1$, then $f(x) \leq \area(C)\bigl(1-\norm{x}(1-1/e)\bigr)$.
The claim follows.
\end{proof}

We show now that the number of ranges induced by translations of $C$ on certain random sets of $n$ points is $\tilde\Omega(n^2)$.
Because this is meant to be used in the context of $T_C$-$k$-sets, we show a slightly stronger bound for ranges (induced by translations) that \emph{do not contain points on their boundaries}.
\begin{lemma}[lower bound on number of ranges, probabilistic]\label{lem:lowerboundrandom}
Let $C \subseteq \RR^2$ be the interior of a convex body.
Let $A \subseteq \RR^2$ be a compact set such that $2C \subseteq A$.\details{Where is this assumption used? Ans: It guarantees that the translations of $C$ from our grid $G$ in the proof are contained in $A$ so that the probability computation goes through.}
Let $X$ be a set of $n$ iid uniformly random points in $A$.
Let $t > 0$.
Then there exists a constant $c_{\ref*{lem:lowerboundrandom}}>0$ that depends only on $A$, $C$ and $t$ such that with probability at least $1 -1/ n^t$, 
\[
    \bigl\lvert \{X \cap (x+C) \suchthat x+C\subseteq A, X \cap \bd(x+C) = \emptyset\} \bigr\rvert \geq c_{\ref*{lem:lowerboundrandom}} \left(\frac{n}{\log n} \right)^2.
\]
\end{lemma}
\begin{proof}
%Let $B$ be a maximum area inscribed ellipse in $C$.
%Without loss of generality (up to an invertible affine transformation) $B$ is the ball with center 0 and radius 1.
Let $X = \{X_1, \dotsc, X_n\}$.
Let $B$ be the ball with center 0 and radius 1.
Without loss of generality (up to scaling and translation), $B \subseteq C$.
%In this case, $B$ contains a square of side one centered at the origin.

We will first construct a packing of $n^2/(c \log n)^2$ translations of $C$ with centers in $B$ with area of pairwise symmetric difference at least about $\log(n)/n$, for some $c>0$ to be determined later. 
Let $G$ be an $n/\gap$-by-$n/\gap$ grid of points with gap $\gap/n$ between adjacent rows and columns of points and contained in $B$.
Every pair of points in $G$ is then at distance at least $\gap/n$, and therefore for all $x, y \in G$ with $0 < \norm{x-y} \leq 1$ we have $\area \bigl((x+C)\Delta (y+C)\bigr) = 2\Bigl(\area(C)-\area \bigl(C \cap ((y-x)+C)\bigr) \Bigr) \geq \area(C) \norm{y-x} \geq \area(C) \gap/n$ (using \cref{lem:distance}).
The bound extends to all $x,y \in G$ with $x\neq y$ by monotonicity.

We will now show that with probability at least $1-o(1)$ each $x+C$ with $x \in G$ induces a different range on $X$.
It is enough to show that for all $x, y \in G$ with $x \neq y$ we have $\bigl((x+C)\Delta (y+C)\bigr) \cap X \neq \emptyset$.
Setting $c = (t+2) \area(A)/\area(C)$ , the probability of this event for some $x,y$ is
\begin{align*}
    \pr\bigl((\forall i \in [n]) X_i \notin (x+C)\Delta (y+C)\bigr)
    &\leq \left(1-\frac{\area \bigl((x+C)\Delta (y+C)\bigr)}{\area(A)}\right)^n \\
    &\leq \left(1-\frac{\area(C)}{\area(A)} \frac{\gap}{n}\right)^n \\
    &\leq  e^{-(t+2)\log n} \\
    &=1/n^{t+2}.
\end{align*}
Thus, with probability at least $1-n^2 / n^{t+2}= 1-1/n^t$ our event holds for all pairs $x,y$.

Finally, $X \cap \cup_{x\in G} \bd(x+C) = \emptyset$ a.s. The claim follows.
\end{proof}

We now state and prove our probabilistic lower bound for $T_C$-$k$-sets for some $k$ proportional to $n$:
\begin{theorem}\label{thm:vcavg}
% Let $(X,R)$ be a set system with growth function $s(n) \leq cn^t$ for fixed $t \geq 1$, $c \geq 1$.
%Let $(X,\beta,\pdist)$ be a probability space with $R\subseteq \beta$ and all events in $R$ of the same probability $p \in (0,1)$.
Let $C \subseteq \RR^2$ be the interior of a convex body.
Let $A \subseteq \RR^2$ be a compact set such that $2C \subseteq A$.
%Let $X = \{X_1, \dotsc, X_n\}$ be an iid.\ uniformly random sample in $A$.
Let $X$ be a set of $n$ iid uniformly random points in $A$.
%Let $\mathcal{D}$ be the uniform distribution in a compact set $A \subseteq \RR^2$ such that $2C \subseteq A$.
%Let $X = \{X_1, \dotsc, X_n\}$ be an iid.\ sample according to $\mathcal{D}$.
Let 
\[
    a'_{k}(X) := \bigl\lvert\{X \cap (x+C) \suchthat x+C\subseteq A, \card{X \cap (x+C)}=k, X \cap \bd(x+C) = \emptyset\}\bigr\rvert
\]
(that is, $a'_{k}(X)$ is the number of $T_C$-k-sets of $X$ induced by translations of $C$ contained in $A$).
Let $p = \area(C)/\area(A)$.
Then there exists a function $k(n)$ such that $\e\bigl(a_{k(n)}(X)\bigr) \geq \e\bigl(a'_{k(n)}(X)\bigr) \geq \Omega(n^{3/2}/(\log n)^{5/2})$ and $\abs{k(n) - pn} \leq O(\sqrt{n \log n})$ (where function $k(n)$ and the constants in $O, \Omega$ depend only on $A$ and $C$).
\end{theorem}
\begin{proof}
Let $(A,\mathcal{R})$ be the set system where $\mathcal{R}$ is the family of translations of $C$ contained in $A$.
From \cref{thm:growth} we have that the growth function $s(n)$ of $(A,\mathcal{R})$ satisfies $s(n) \leq n^3$.

Fix $n$. 
%Let $Q \subseteq X$ be a set of $n$ iid random points according to $\pdist$. 
For $R \in \mathcal{R}$, let $\hat\pdist(R) = \card{X \cap R}/n$.
From \cite[Theorem 2]{VapChe71} (VC's uniform convergence theorem)\footnote{In order to apply VC's uniform convergence theorem, we need to verify that the function $\sup_{R \in \mathcal{R}}(\cdot)$ as defined in \cref{eq:good} is measurable, i.e., that it is a random variable. This can be verified by observing that $\mathcal{R}$ is a \emph{permissible} class of subsets of $A$. See \cite[Appendix C]{Pollard} for the definition of permissible classes and a proof of the measurability of suprema in this context. One can see that the class $\mathcal{R}$ is permissible by indexing it by translation and verifying that the requirements for permissibility are met.} we have, for $n \geq 2/\eps^2$,
\begin{equation}\label{eq:good}
\pr \biggl(\sup_{R \in \mathcal{R}} \abs{\hat\pdist(R) - p} > \eps \biggr) \leq 4 s(2n) e^{-\eps^2 n/8}.
\end{equation}
Set $\eps = 4\sqrt{\frac{3 \log 2n }{n}}$ so that the rhs is at most $1/n^3$.
If we denote by $G$ the complement of the event in \eqref{eq:good}, we have $\pr(G) = 1-o(1)$.

Let $\mathcal{R}_X = \{X \cap R \suchthat R \in \mathcal{R}, X \cap \bd R = \emptyset\}$. 
From \cref{lem:lowerboundrandom} with $t=1$ and notation $f(n) = c_{\ref*{lem:lowerboundrandom}}(n/\log n)^2$, we get
\begin{equation}\label{eq:lowerboundrandom}
    \pr \bigl(
        \card{\mathcal{R}_X} \geq f(n) 
        \bigr) \geq 1-o(1).
\end{equation}
Let $H$ denote the event in \eqref{eq:lowerboundrandom}.

\details{Old pigeonhole argument:
For $X \in G$ and for all $S \in R$ we have that $\abs{\card{S \cap X} - pn} \leq n \eps = O(\sqrt{n \log n})$.
For $X \in H$ there are at least $f(n)$ choices of $S \cap X$ with $S \in R$.
Therefore, for $X \in G \cap H$ and by the pigeonhole principle, over the choices of $S \cap X$ there is a value of $\card{S \cap X}$, say $k(X)$, that appears at least $\Omega( f(n)/ \sqrt{n \log n })$ times. 
The value $k(X)$ is such that $\abs{k - pn} \leq O(\sqrt{n \log n})$.

}
To conclude, we will show that there is a value $k(n)$ that is independent of $X$ and makes $\e\bigl(a'_k(X) \bigr)$ large.
We have $\pr(X \in G \cap H) = 1- o(1)$.
Also for $X \in G \cap H$ we have 
\[
\sum_{k \in [pn - n\eps,pn+n\eps]} a'_k(X) = \card{\mathcal{R}_X} \geq f(n).
\]
Therefore 
\[
    \e\left(\sum_{k \in [pn - n\eps,pn+n\eps]} a'_k(X) \right) \geq f(n) \pr(X \in G \cap H).
\]
Reordering,
$\sum_{k \in [pn - n\eps,pn+n\eps]} \e\bigr( a'_k(X) \bigr) \geq f(n) \pr(X \in G \cap H)$.
Thus, there exists $k(n) \in [pn - n\eps,pn+n\eps]$ such that
$\e\bigr( a'_k(X) \bigr) \geq f(n) \pr(X \in G \cap H)/(2n\eps+1)$.
That is (using $n\eps = O(\sqrt{n \log n})$),
$\e\bigr( a'_k(X) \bigr) \geq \Omega \bigl( n^{3/2}/(\log n)^{5/2} \bigr) $.

For the bound on $\e \bigl(a_k(X) \bigr)$, from the definitions we have $a_k(X) \geq a'_k(X)$.
\end{proof}
    
\iffalse    
\begin{theorem}
Let $\mathcal{D}$ be the uniform distribution in a compact set $A \subseteq \RR^2$ such that $2C \subseteq A$.
Let $X$ be a sample of $n$ points according to $\mathcal{D}$. Then the expected number of $T_C$-$k$-edges ($k=\frac{n-2}{2}$) of $X$ is $\Theta(n^{3/2})$. 
\end{theorem}
\begin{proof}
Let $x_1,x_2$ be two random points from $\mathcal{D}$. Let $C(
x_1,x_2)$ be either of the two translates of $C$ that contains $x_1,x_2$ on their boundary (if they exist). Use $C(
x_1,x_2)^+$ to denote the interior of $C(x_1,x_2)$. Define $G_P(t) = \mathbb{P}(\mathcal{D}(C(x_1,x_2)^+)\le t)$. 
Then the expected number of $T_C$-$k$-edges of $X$ is at most
\begin{equation}
\begin{split}
&2\sum_{(x_1,x_2) \in \binom{X}{2}} \mathbb{P}(C(x_1,x_2)^+ \text{ contains exactly }k \text{ points of } X) \\
&= 2\binom{n}{2}\binom{n-2}{k} \int_0^1 t^k(1-t)^{n-d-k}dG_P(t) \\
&= O(n^{3/2}).
\end{split}
\end{equation}

\end{proof}
\fi

\section{Conclusion and open questions}\label{sec:conclusion}
We conclude with some open questions and possible directions for future work.
\begin{itemize}
\item
\textbf{Bound on the number of points of intersection of $Z(f)$ and $G_k$.}
Can the answer to \cref{question} given in \cref{lem:intersections} be improved? We believe it may be possible to improve the bound to $O(nr)$ for the following reason: The $k$-edge graph $G_k$ of a set of $n$ points behaves somewhat like a degree $n$ algebraic curve when it comes to intersecting it with a line. In particular, a degree $n$ algebraic curve and the $k$-edge graph of a set of $n$ points both have the property that a line can intersect them at most $n$ times unless the line intersects them infinitely many times. One might expect this phenomenon to also hold for arbitrary algebraic curves, not just lines. If this is the case, the bound on the number of intersection points between a degree $r$ algebraic curve and the $k$-edge graph of a set of $n$ points should be $O(nr)$ as in the case of the intersection of a degree $r$ algebraic curve and a degree $n$ algebraic curve.
\item
\textbf{Higher dimensions.}
The polynomial partitioning theorem becomes more powerful in higher dimensions. It may be possible to apply it to the $k$-set problem in dimensions higher than $2$. The issue is that there is no analogue of the convex chains decomposition in dimension higher than $2$, so this would likely require the discovery of a new property of $k$-sets or $k$-facets. 
%\item\textbf{Distribution of the number of $k$-sets as a function of $k$.}

\end{itemize}

\begin{acknowledgements}
We would like to thank the reviewers for their helpful suggestions. In particular, we are grateful to the anonymous reviewer who noticed that the bound in \cref{thm:5/4} could be significantly improved with a simple modification to the proof. This material is based upon work supported by the National Science Foundation under Grants CCF-2006994, CCF-1657939, CCF-1422830 and CCF-1934568.
\end{acknowledgements}

%\paragraph{Data Availability Statement.} Data sharing not applicable to this article as no datasets were generated or analyzed during the current study.

\bibliographystyle{abbrv}
\bibliography{bib}

\end{document}